\documentclass[12pt,thmsa]{article}
\usepackage{amssymb}
\usepackage{amsfonts}
\usepackage{enumerate}
\usepackage{graphicx}
\usepackage{subfigure}
\usepackage{float}
\usepackage{color}
\usepackage{amsmath}
\usepackage{cite}
\usepackage[top=3.8cm,bottom=3.8cm,textwidth=16cm,centering,a4paper]{geometry}

\newtheorem{theorem}{Theorem}[section]

\newtheorem{remark}{Remark}[section]

\newtheorem{lemma}[theorem]{Lemma}

\newenvironment{proof}[1][Proof]{\noindent\textbf{#1.} }{\ \rule{0.5em}{0.5em}}

\makeatletter
\@addtoreset{equation}{section}
\makeatother

\begin{document}

\title{On indefinite Kirchhoff-type equations under the combined effect of
linear and superlinear terms}
\date{}
\author{Juntao Sun$^{a}$\thanks{%
E-mail address: jtsun@sdut.edu.cn(J. Sun)}, Kuan-Hsiang Wang$^{b}$\thanks{%
E-mail address: khwang0511@gmail.com (K.-H. Wang)}, Tsung-fang Wu$^{b}$%
\thanks{%
E-mail address: tfwu@nuk.edu.tw (T.-F. Wu)} \\
{\footnotesize $^a$\emph{School of Mathematics and Statistics, Shandong
University of Technology, Zibo, 255049, P.R. China }}\\
{\footnotesize $^{b}$\emph{Department of Applied Mathematics, National
University of Kaohsiung, Kaohsiung 811, Taiwan }}}
\maketitle

\begin{abstract}
We investigate a class of Kirchhoff type equations involving a combination
of linear and superlinear terms as follows:
\begin{equation*}
-\left( a\int_{\mathbb{R}^{N}}|\nabla u|^{2}dx+1\right) \Delta u+\mu
V(x)u=\lambda f(x)u+g(x)|u|^{p-2}u\quad \text{ in }\mathbb{R}^{N},
\end{equation*}%
where $N\geq 3,2<p<2^{\ast }:=\frac{2N}{N-2}$, $V\in C(\mathbb{R}^{N})$ is a
potential well with the bottom $\Omega :=int\{x\in \mathbb{R}^{N}\ |\
V(x)=0\}$. When $N=3$ and $4<p<6$, for each $a>0$ and $\mu $ sufficiently
large, we obtain that at least one positive solution exists for $%
0<\lambda\leq\lambda _{1}(f_{\Omega}) $ while at least two positive
solutions exist for $\lambda _{1}(f_{\Omega })< \lambda<\lambda
_{1}(f_{\Omega})+\delta_{a}$ without any assumption on the integral $%
\int_{\Omega }g(x)\phi _{1}^{p}dx$, where $\lambda _{1}(f_{\Omega })>0$ is
the principal eigenvalue of $-\Delta $ in $H_{0}^{1}(\Omega )$ with weight
function $f_{\Omega }:=f|_{\Omega }$, and $\phi _{1}>0$ is the corresponding
principal eigenfunction. When $N\geq 3$ and $2<p<\min \{4,2^{\ast }\}$, for $%
\mu $ sufficiently large, we conclude that $(i)$ at least two positive
solutions exist for $a>0$ small and $0<\lambda <\lambda _{1}(f_{\Omega })$; $%
(ii)$ under the classical assumption $\int_{\Omega }g(x)\phi _{1}^{p}dx<0$,
at least three positive solutions exist for $a>0$ small and $\lambda
_{1}(f_{\Omega })\leq \lambda<\lambda _{1}(f_{\Omega})+\overline{\delta }%
_{a} $; $(iii)$ under the assumption $\int_{\Omega }g(x)\phi _{1}^{p}dx>0$,
at least two positive solutions exist for $a>a_{0}(p)$ and $\lambda^{+}_{a}<
\lambda<\lambda _{1}(f_{\Omega})$ for some $a_{0}(p)>0$ and $\lambda^{+}_{a}\geq0$.
\end{abstract}


\section{Introduction}

In this paper, we investigate the following Kirchhoff type equation:
\begin{equation}
\left\{
\begin{array}{ll}
-\left( a\int_{\mathbb{R}^{N}}|\nabla u|^{2}dx+b\right) \Delta u+\mu
V(x)u=h(x,u) & \text{ in }\mathbb{R}^{N}, \\
u\in H^{1}(\mathbb{R}^{N}), &
\end{array}%
\right.  \label{1-1}
\end{equation}%
where $N\geq 3$, the parameters $a,b,\mu >0$, and the potential $V$
satisfies the following conditions:

\begin{itemize}
\item[$(V1)$] $V\in C(\mathbb{R}^{N})$ with $V(x)\geq 0$ in $\mathbb{R}^{N}$
and there exists $c_{0}>0$ such that the set $\{V<c_{0}\}:=\{x\in \mathbb{R}%
^{N}\ |\ V(x)<c_{0}\}$ is nonempty and has finite positive measure.

\item[$(V2)$] $\Omega :=int\{x\in \mathbb{R}^{N}\ |\ V(x)=0\}$ is nonempty
bounded domain and has a smooth boundary with%
\begin{equation*}
\overline{\Omega }=\{x\in \mathbb{R}^{N}\ |\ V(x)=0\}.
\end{equation*}
\end{itemize}

The Kirchhoff type equation is an extension of the classical D'Alembert wave
equation, namely,
\begin{equation}
u_{tt}-M\left( \int_{D}|\nabla u|^{2}dx\right) \Delta u=h(x,u),  \label{1-3}
\end{equation}%
which was proposed by Kirchhoff \cite{K} to describe the transversal
oscillations of a stretched string, taking into account of the effect of
changes in string length during the vibrations. Here $D$ is a bounded
domain, $u$ denotes the displacement and $h$ is the external force. In
particular, if $M(t)=at+b$, then $b$ denotes the initial tension while $a$
is related to the intrinsic properties of the string (such as Young's
modulus). It is notable that Eq. $(\ref{1-3})$ is often referred to as being
non-local because of the presence of the integral over the domain $D$. About
the solvability of Eq. $(\ref{1-3}),$ we refer to the reader to the papers
\cite{AP,D'A-S,L,P}.

In recent years, the stationary analogue of Eq. $(\ref{1-3})$ with specific
formulations of $M$ and $h$ has been widely studied in bounded domain \cite%
{CKW,D,JS,LLS,SW2,W1} and in unbounded domain \cite%
{AF,C,DPS,FIJ,G,HZ,I,JL,LY,SCWF,SW1,SW3,TC,ZD,ZLW} via variational methods.
Now let us briefly comment some known results releted to our work.

Sun and Wu \cite{SW1}, the first and third authors of the current paper,
studied the existence of ground state solution for Eq. $(\ref{1-1})$ with $%
N=3,$ where $V$ satisfies conditions $(V1)-(V2)$ and the nonlinearity $h$ is
required to be asymptotically linear, asymptotically $3$-linear and
asymptotically $4$-linear at infinity on $u,$ respectively. The proof is
based on mountain pass theorem and the Nehari manifold method. It is worth
noting that the potential $\mu V,$ first introduced by Bartsch and Wang \cite%
{BW}, is usually called the steep potential well whose depth is controlled
by the parameter $\mu $. Later, the corresponding results were further
extended and improved by Jia and Luo \cite{JL} and Zhang and Du \cite{ZD}.

Recently, Sun et al. \cite{SCWF} considered the case of $N\geq 4.$ They
found that when $h$ is superlinear and subcritical on $u$, the geometric
structure of the functional $J$ related to Eq. $(\ref{1-1})$ is known to
have a global minimum and a mountain pass, due to the forth power of the
non-local term. As a result, two positive solutions of Eq. $(\ref{1-1})$ can
be found. After that, Sun and Wu \cite{SW3} showed that when $N=3$ and $%
h(x,u)=g(x)|u|^{p-2}u$ with $2<p<4$, the functional $J$ related to Eq. $(\ref%
{1-1})$ also has a global minimum and a mountain pass. However, an
additional assumption on $V$ and $g$ needs to be required as follows:\newline
$\left( H3\right) $ There exist two numbers $c_{\ast },R_{\ast }>0$ such that%
\begin{equation*}
\left\vert x\right\vert ^{p-2}g(x)\leq c_{\ast }\left[ V\left( x\right) %
\right] ^{4-p}\text{ for all }\left\vert x\right\vert >R_{\ast }.
\end{equation*}

We notice that there seems to be rarely concerned on Kirchhoff type equation
involving a combination of linear and superlinear terms in the existing
literature. We are only aware of the work \cite{ZLW}. Actually, the combined
effect of linear and superlinear terms was first studied by Alama and
Tarantello \cite{AT} in the following indefinite semilinear elliptic
equations in bounded domain:

\begin{equation*}
\left\{
\begin{array}{ll}
-\Delta u=\lambda u+a(x)h(u) & \quad \text{in }D\text{$,$} \\
u=0 & \quad \text{on $\partial D,$}%
\end{array}%
\right.
\end{equation*}%
where $a\in C(\overline{D})$ changes sign in $D$ and $h$ is a nonlinear
function with superquadratic growth both at zero and at infinity. They
concluded that for $\lambda $ in a small right neighborhood of $\lambda _{1}$%
, the first eigenvalue of $-\Delta $ in $H_{0}^{1}(D)$, the condition $%
\int_{D}a(x)\phi _{1}^{p}dx<0$ is necessary and sufficient for existence of
a positive solution, where $\phi _{1}>0$ is is the corresponding principal
eigenfunction. Furthermore, the existence of two positive solutions for $%
\lambda \in (\lambda _{1},\gamma )$, for some $\gamma >\lambda _{1},$ is
also established in \cite{AT}. For more similar results, we refer the reader
to \cite{AG,BZ,CT,HRC}.

Very recently, Zhang et al. \cite{ZLW} extended the analysis to Kirchhoff
type equation with a combination of linear and superlinear terms, namely,
Eq. $(\ref{1-1})$ with $h(x,u)=\lambda f(x)u+g(x)|u|^{p-2}u,$ where $%
2<p<2^{\ast }$ and $f,g$ are both sign-changing in $\mathbb{R}^{N}$. They
illustrated the difference in the solution behavior which arises from the
consideration of the nonlocal and eigenvalue problem effects. By using the
Nehari manifold method and giving an approximation estimate of eigenvalue
problem, they explored the existence and multiplicity of positive solutions
when $\lambda $ lies in the left and right neighborhood of $\lambda
_{1}(f_{\Omega })$, respectively, where $\lambda _{1}(f_{\Omega })$ is the
positive principal eigenvalue of the problem%
\begin{equation}
\left\{
\begin{array}{ll}
-\Delta u=\lambda f_{\Omega }(x)u & \quad \text{in $\Omega ,$} \\
u=0 & \quad \text{on $\partial \Omega ,$}%
\end{array}%
\right.  \label{eb}
\end{equation}%
with $\phi _{1}$ the corresponding positive principal eigenfunction, here $%
\Omega $ is given in condition $(V2)$ and the function $f_{\Omega }$ is a
restriction of $f$ on $\Omega $. Some of the results obtained in \cite{ZLW}
are summarized as follows (also see Figure 1):

\begin{figure}[H]
\centering
\subfigure[$N=3, 4<p<6$ and $\int_{\Omega }g(x)\phi _{1}^{p}dx<0$]{
        \includegraphics[width=2.5in,height=1.6in]{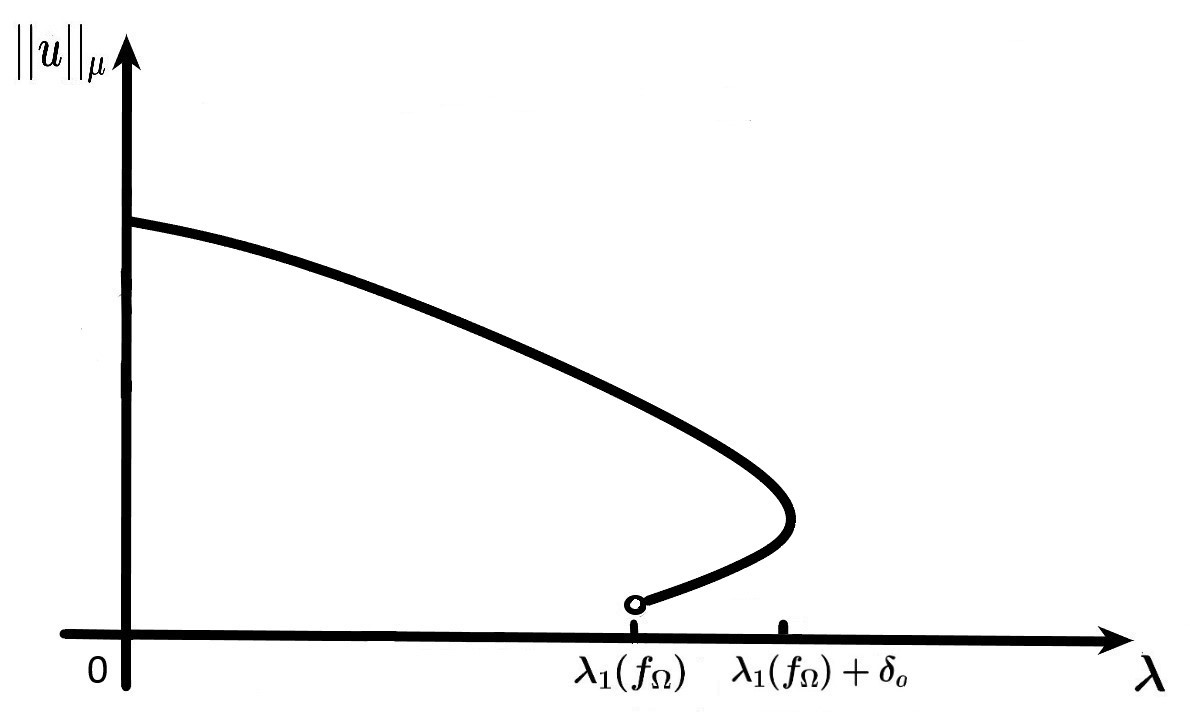}}
\subfigure[$N\geq3$ and $2<p<\min \{4,2^{\ast }\}$]{
        \includegraphics[width=2.5in,height=1.6in]{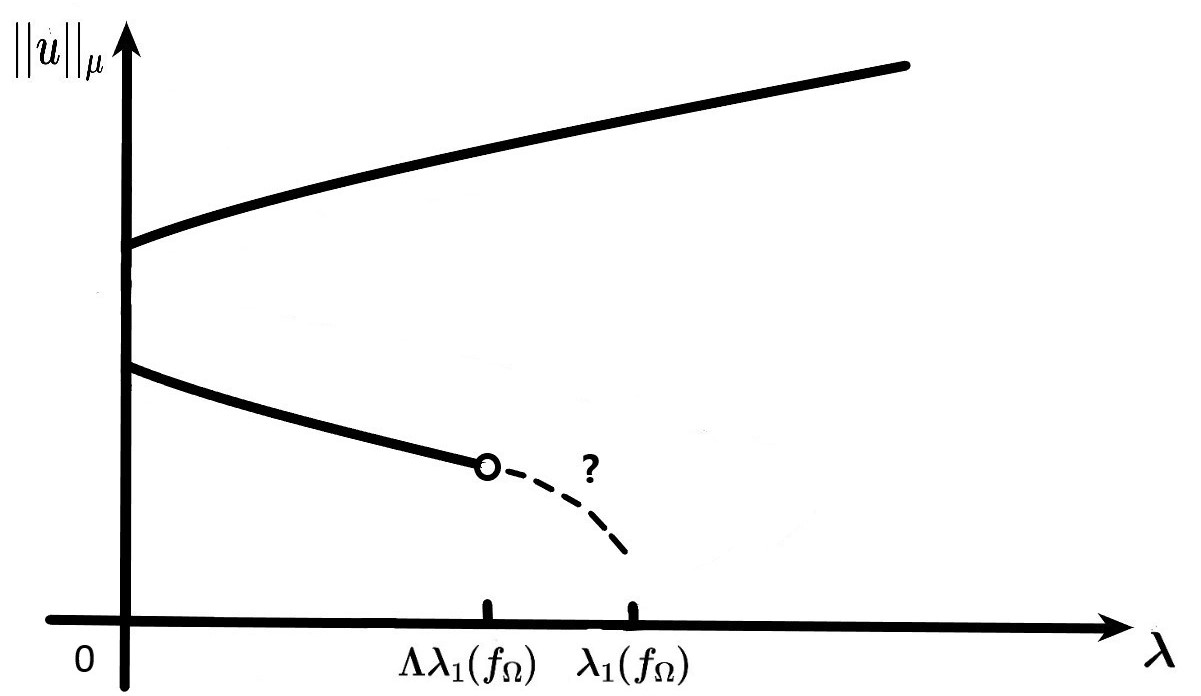}}
\caption{Bifurcation diagrams for result $(i)$ on (a) and for results $%
(ii)-(iii)$ on (b).}
\end{figure}

\begin{itemize}
\item[$(i)$] $N=3$ and $4<p<6:$ if $\int_{\Omega }g(x)\phi _{1}^{p}dx<0$,
then for each $a>0$ there exists a number $\delta _{0}>0$ such that for each
$\lambda _{1}(f_{\Omega })<\lambda <\lambda _{1}(f_{\Omega })+\delta _{0}$,
Eq. $(\ref{1-1})$ admits at least two positive solutions for $\mu $
sufficiently large;

\item[$(ii)$] $N\geq 3$ and $2<p<\min \{4,2^{\ast }\}:$ there exists a
number $a_{0}>0$ such that for each $0<a<a_{0}$ and $0<\lambda <\Lambda
\lambda _{1}(f_{\Omega })$, Eq. $(\ref{1-1})$ admits at least two positive
solutions for $\mu $ sufficiently large, where
\begin{equation*}
\Lambda :=1-2\left( \frac{4-p}{4}\right) ^{2/p}<1\text{ for }2<p<\min
\{4,2^{\ast }\};
\end{equation*}

\item[$(iii)$] $N\geq 3$ and $2<p<\min \{4,2^{\ast }\}:$ there exists a
number $a_{0}>0$ such that for each $0<a<a_{0}$ and $\lambda \geq \Lambda
\lambda _{1}(f_{\Omega })$, Eq. $(\ref{1-1})$ admits at least a positive
solution for $\mu $ sufficiently large.
\end{itemize}

From the results mentioned above, it is very natural for us to raise a
series of interesting questions, such as the following

\begin{itemize}
\item[$(I)$] The condition $\int_{\Omega }g(x)\phi _{1}^{p}dx<0$ appears
necessary in finding two positive solutions when $N=3$ and $4<p<6$ in \cite%
{ZLW}, as well as in the study of local elliptic equations \cite{AT,AG,BZ,CT}%
. Can one obtain the same result as described in \cite{ZLW} without this
condition?

\item[$(II)$] When $N\geq 3$ and $2<p<\min \{4,2^{\ast }\}$ in \cite{ZLW},
the existence of two positive solutions is established only in the range of $%
0<\lambda <\Lambda \lambda _{1}(f_{\Omega }),$ while not including the range
of $\Lambda \lambda _{1}(f_{\Omega })\leq \lambda \leq\lambda _{1}(f_{\Omega
}). $ In view of this, we wonder if two positive solutions can be found when
$\Lambda \lambda _{1}(f_{\Omega })\leq \lambda \leq\lambda _{1}(f_{\Omega
}), $ like that in the case of $0<\lambda <\Lambda \lambda _{1}(f_{\Omega }).
$

\item[$(III)$] It is notable that Zhang et al. \cite{ZLW} only found one
positive solution when $\lambda \geq \Lambda \lambda _{1}(f_{\Omega })$ for $%
N\geq 3$ and $2<p<\min \{4,2^{\ast }\}.$ In other words, they can not
conclude that $\lambda =\lambda _{1}(f_{\Omega })$ is a bifurcation point of
Eq. $(\ref{1-1})$ with positive solutions bifurcating to the right of $%
\lambda _{1}(f_{\Omega }).$ Based on $(II)$, we would like to further probe
into whether there exists a bifurcation phenomenon at the point $\lambda
=\lambda _{1}(f_{\Omega }).$
\end{itemize}

In the present paper, we are very interested in seeking definite answers to
Questions $(I)-(III)$ and establishing the multiplicity of positive
solutions for Eq. $(\ref{1-1})$ with $h(x,u)=\lambda f(x)u+g(x)|u|^{p-2}u$
by using the mountain pass theory and the direct sum decomposition of the
function. Here we wish to point out that the Nehari manifold method used in
\cite{ZLW} is not a good choice in our study. Indeed, for $N=3$ and $4<p<6$,
the condition $\int_{\Omega }g(x)\phi _{1}^{p}dx<0$ is used to ensure that
the Nehari manifold is a natural constraint and can be decomposed into two
nonempty submanifolds. Moreover, for $N\geq 3$ and $2<p<\min \{4,2^{\ast }\}$%
, the filtration of the Nehari manifold is adopted to derive the boundedness
of (PS)-sequence, which is available only for $0<\lambda <\Lambda \lambda
_{1}(f_{\Omega })$. For simplicity, we always assume that $b=1$ in Eq. $(\ref%
{1-1}).$ The problem we consider is thus
\begin{equation}
\left\{
\begin{array}{ll}
-\left( a\int_{\mathbb{R}^{N}}|\nabla u|^{2}dx+1\right) \Delta u+\mu
V(x)u=\lambda f(x)u+g(x)|u|^{p-2}u & \text{ in }\mathbb{R}^{N}, \\
u\in H^{1}(\mathbb{R}^{N}), &
\end{array}%
\right.  \tag{$K_{a,\lambda }^{\mu }$}
\end{equation}%
where $N\geq 3,2<p<2^{\ast }$, the parameters $a,\mu ,\lambda >0$, the
potential $V$ satisfies conditions $(V1)-(V2)$, and the weight functions $%
f,g $ satisfy the following conditions:

\begin{itemize}
\item[$( D1) $] $f\in L^{N/2}(\mathbb{R}^{N})\cap L^{\infty }(\Omega )$ and $%
f^{+}:=\max \{f,0\}\not\equiv 0$ in $\Omega ;$

\item[$(D2) $] $g\in L^{\infty }(\mathbb{R}^{N})$ and $g^{+}\not\equiv 0$ in
$\Omega .$
\end{itemize}

\begin{remark}
\label{R1.1}If $f$ is bounded in $\Omega $ and $|\{x\in \Omega \ |\
f(x)>0\}|>0$, then there exists a sequence of eigenvalues $\{\lambda
_{n}(f_{\Omega })\}$ of Eq. $(\ref{eb})$ with $0<\lambda _{1}(f_{\Omega
})<\lambda _{2}(f_{\Omega })\leq \cdots $ and each eigenvalue being of
finite multiplicity. Denoting the positive principal eigenfunction by $\phi
_{1}$, we have
\begin{equation*}
\lambda _{1}(f_{\Omega })=\int_{\Omega }|\nabla \phi _{1}|^{2}dx=\inf
\left\{ \int_{\Omega }|\nabla u|^{2}dx\ |\ u\in H_{0}^{1}(\Omega
),\int_{\Omega }f_{\Omega }(x)u^{2}dx=1\right\}
\end{equation*}%
and
\begin{equation*}
\lambda _{2}(f_{\Omega })=\inf \left\{ \int_{\Omega }|\nabla u|^{2}dx\ |\
u\in H_{0}^{1}(\Omega ),\int_{\Omega }f_{\Omega }(x)u^{2}dx=1,\int_{\Omega
}\nabla u\nabla \phi _{1}dx=0\right\} .
\end{equation*}
\end{remark}

We now summarize our main results as follows.

\begin{theorem}
\label{T1} Suppose that $N=3,4<p<6$ and conditions $(V1)-(V2),(D1)-(D2)$
hold. Then for each $a>0$, the following statements are true.\newline
$(i)$ For each $0<\lambda \leq \lambda _{1}(f_{\Omega })$, Eq. $%
(K_{a,\lambda }^{\mu })$ has at least one positive solution for $\mu $
sufficiently large.\newline
$(ii)$ There exists a number $\delta _{a}>0$ such that for every $\lambda
_{1}(f_{\Omega })<\lambda <\lambda _{1}(f_{\Omega })+\delta _{a}$, Eq. $%
(K_{a,\lambda }^{\mu })$ has at least two positive solutions for $\mu $
sufficiently large.
\end{theorem}

\begin{remark}
\label{R1} $(i)$ Theorem \ref{T1} $(ii)$ give an answer to Question $(I).$%
\newline
$(ii)$ In fact, in the proof of Theorem \ref{T1} $(ii)$ one can see that the
length of right neighborhood of $\lambda _{1}(f_{\Omega }),$ i.e. $\delta
_{a}$ can be given explicitly by
\begin{equation*}
\delta _{a}=\min \left\{ a^{\frac{p-2}{p-4}}C_{1},C_{2}\right\} ,
\end{equation*}%
where $C_{1},C_{2}>0$. It means that $\delta _{a}$ depends on the parameter $%
a$ when $a$ is small while not depending on it when $a$ is large.
\end{remark}

Let us consider the following problem:
\begin{equation*}
\Gamma _{p}:=\sup \left\{ \frac{\int_{\Omega }g(x)|u|^{p}dx}{\left(
\int_{\Omega }|\nabla u|^{2}dx\right) ^{p/2}}\ |\ u\in H_{0}^{1}(\Omega
)\diagdown \{0\},\int_{\Omega }f_{\Omega }(x)u^{2}dx\geq 0\right\} .
\end{equation*}%
Under conditions $(D1)-(D2)$, we can choose a function $\varphi \in
H_{0}^{1}(\Omega )$ such that $\int_{\Omega }f_{\Omega }(x)\varphi ^{2}dx>0$
and $\int_{\Omega }g(x)|\varphi |^{p}dx>0$ (see \cite[Proposition 6.2]{CC}
for more details). Then, it is easy to deduce that $0<\Gamma _{p}<\infty $
by conditions $(V2),(D2)$ and Sobolev inequality. Now we set
\begin{equation*}
a_{0}(p)=2(p-2)\left( 4-p\right) ^{\frac{4-p}{p-2}}\left( \frac{\Gamma _{p}}{%
p}\right) ^{\frac{2}{p-2}}\text{ for }2<p<\min \{4,2^{\ast }\}.
\end{equation*}

\begin{theorem}
\label{T2} Suppose that $N\geq 3,2<p<\min \{4,2^{\ast }\}$ and conditions $%
(V1)-(V2),(D1)-(D2)$ hold. In addition, for $N=3$, we also assume that
condition $(H3)$ is satisfied. Then for each $0<a<a_{0}(p)$, the following
statements are ture.\newline
$(i)$ For each $0<\lambda <\lambda _{1}(f_{\Omega })$, Eq. $(K_{a,\lambda
}^{\mu })$ has at least two positive solutions for $\mu $ sufficiently large.%
\newline
$(ii)$ If $\int_{\Omega }g(x)\phi _{1}^{p}dx<0$, then we have\newline
$(ii-1)$ for $\lambda =\lambda _{1}(f_{\Omega })$, Eq. $(K_{a,\lambda }^{\mu
})$ has at least two positive solutions for $\mu $ sufficiently large;%
\newline
$(ii-2)$ there exists $\overline{\delta }_{a}>0$ such that for every $%
\lambda _{1}(f_{\Omega })<\lambda <\lambda _{1}(f_{\Omega })+\overline{%
\delta }_{a}$, Eq. $(K_{a,\lambda }^{\mu })$ has at least three positive
solutions for $\mu $ sufficiently large.
\end{theorem}

\begin{remark}
\label{R2} $(i)$ Theorem \ref{T2} answers Questions $(II)$ and $(III).$%
\newline
$(ii)$ Similar to Remark \ref{R1} $(ii)$, $\overline{\delta }_{a}$ can also
be given explicitly by%
\begin{equation*}
\overline{\delta }_{a}=\min \left\{ a^{-\frac{p-2}{4-p}}C_{3},C_{4}\right\} ,
\end{equation*}%
where $C_{3},C_{4}>0$. It shows that $\overline{\delta }_{a}$ does not
depend on $a$ for $a$ small while depending on it for $a$ large.
\end{remark}

To prove Theorems \ref{T1} and \ref{T2}, we need to study the mountain pass
geometry of the functional related to Eq. $(K_{a,\lambda }^{\mu })$ when $%
\lambda $ lies in a right neighborhood of $\lambda _{1,\mu }(f)$ by
decomposing each $u\in X$ (defined later) into the sum of a function in span$%
\{\phi _{1,\mu }\}$ and a function in $\{\text{span}\{\phi _{1,\mu
}\}\}^{\perp }$, where $\lambda _{1,\mu }(f)$ is the positive principal
eigenvalue of the problem
\begin{equation*}
-\Delta u+\mu V(x)u=\lambda f(x)u\quad \text{in $X,$}
\end{equation*}%
and $\phi _{1,\mu }$ is the corresponding positive principal eigenfunction.
Then using the approximation estimate gives
\begin{equation*}
\lambda _{1,\mu }(f)\rightarrow \lambda _{1}^{-}(f_{\Omega })\ \text{as $\mu
\rightarrow \infty $},
\end{equation*}%
which can help us obtain the mountain pass geometry in a right neighborhood
of $\lambda _{1}(f_{\Omega })$. In the case of $N=3$ and $4<p<6$, the
non-local term in Eq. $(K_{a,\lambda }^{\mu })$ can ensure the mountain pass
geometry for $\lambda $ in a right neighborhood of $\lambda _{1}(f_{\Omega
}) $ without any assumption on the integral $\int_{\Omega }g(x)\phi
_{1}^{p}dx$. However, when $N\geq 3$ and $2<p<\min \{4,2^{\ast }\}$, we need
to add the condition $\int_{\Omega }g(x)\phi _{1}^{p}dx<0$ to ensure the
mountain pass geometry for $\lambda $ in a right neighborhood of $\lambda
_{1}(f_{\Omega }).$

Next, we consider the case of $\int_{\Omega }g(x)\phi _{1}^{p}dx>0.$ Let
\begin{equation}
\lambda _{a}^{+}=\lambda _{1}(f_{\Omega })-(4-p)\left( \frac{%
\int_{\Omega}g(x)\phi _{1}^{p}dx}{p}\right) ^{\frac{2}{4-p}}\left( \frac{%
2(p-2)}{a\lambda _{1}^{2}(f_{\Omega })}\right) ^{\frac{p-2}{4-p}}.
\label{lam:a+}
\end{equation}%
Note that $0\leq \lambda _{a}^{+}<\lambda _{1}(f_{\Omega })$ for each $a\geq
a_{0}(p)$. Then we have the following result.

\begin{theorem}
\label{T3} Suppose that $N\geq 3,2<p<\min \{4,2^{\ast }\}$ and conditions $%
(V1)-(V2),(D1)-(D2)$ hold. In addition, for $N=3$, we also assume that
condition $(H3)$ is satisfied. If $\int_{\Omega }g(x)\phi _{1}^{p}dx>0$,
then for each $a\geq a_{0}(p)$ and $\lambda _{a}^{+}<\lambda <\lambda
_{1}(f_{\Omega }),$ Eq. $(K_{a,\lambda }^{\mu })$ has at least two positive
solutions for $\mu $ sufficiently large.
\end{theorem}

\begin{figure}[H]
\centering
\subfigure[$\int_{\Omega }g(x)\phi _{1}^{p}dx<0$ and $a<a_{0}(p)$]{
        \includegraphics[width=2.5in,height=1.6in]{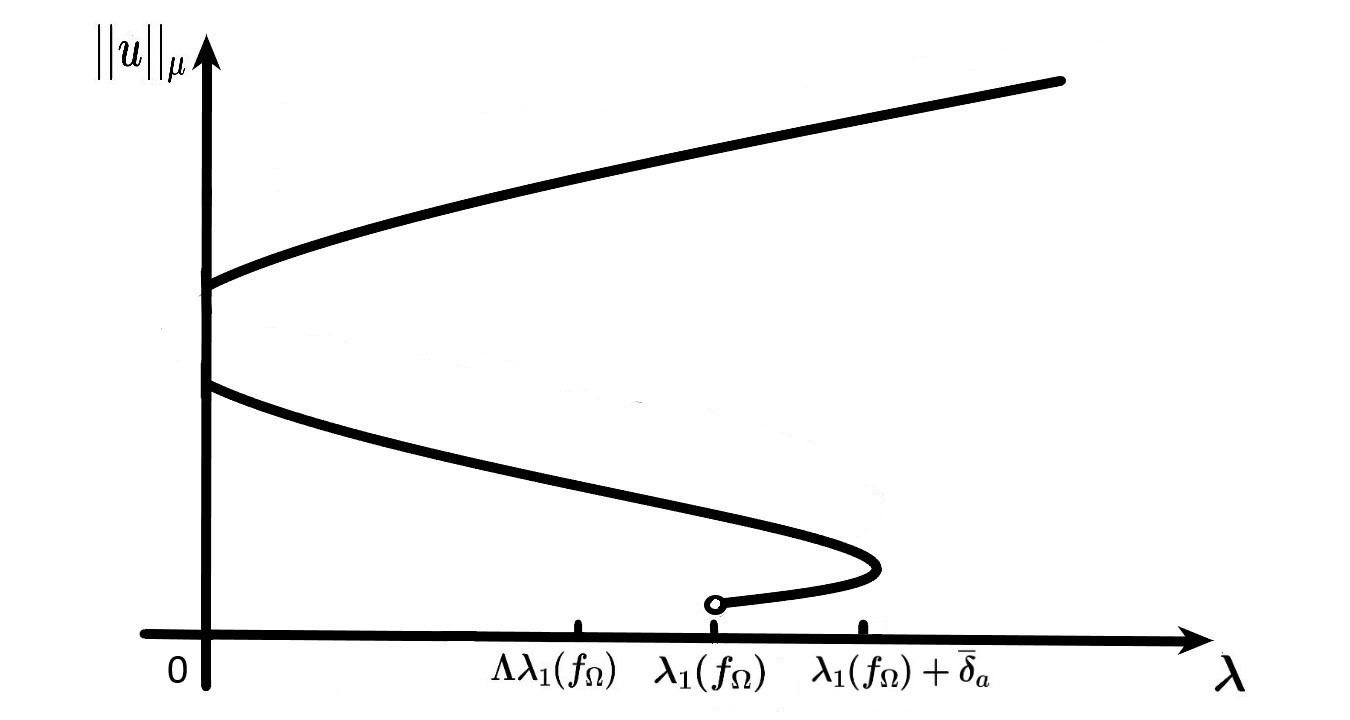}}
\subfigure[$\int_{\Omega }g(x)\phi _{1}^{p}dx>0$ and $a_0(p)<a_1<a_2$]{
        \includegraphics[width=2.5in,height=1.6in]{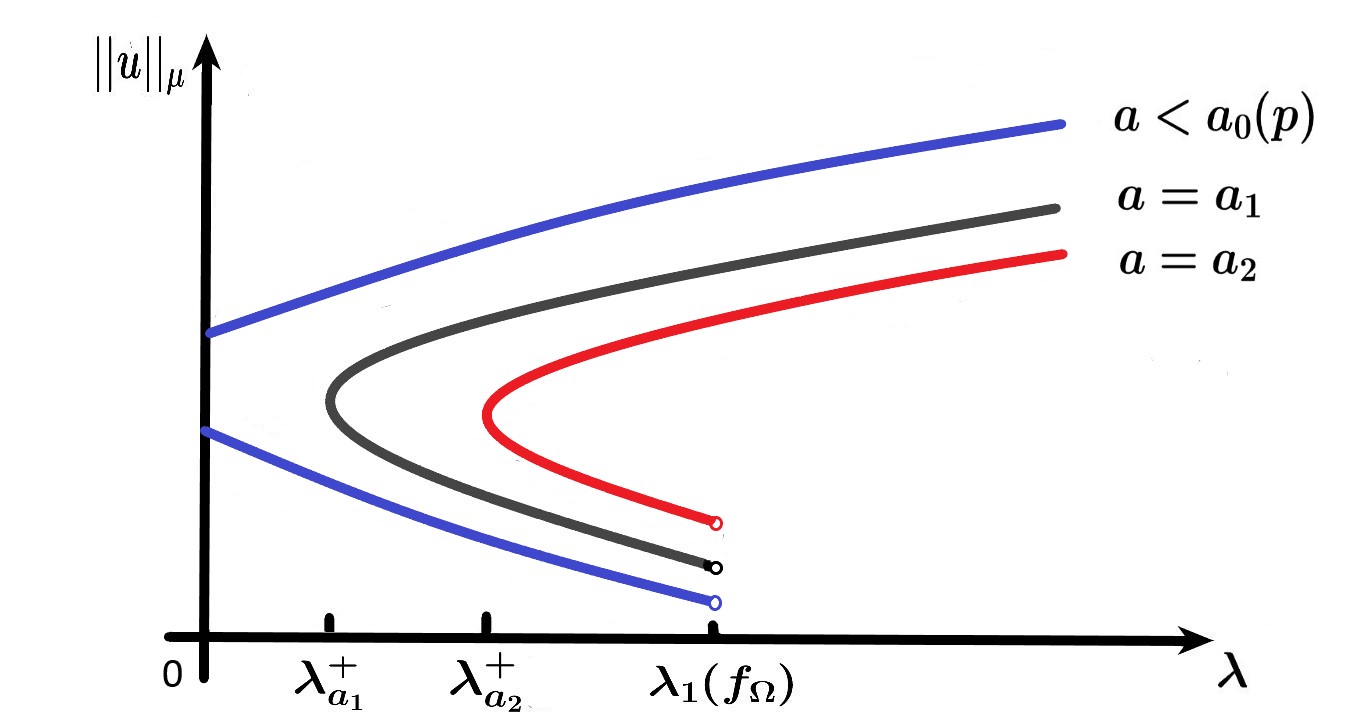}}
\caption{Bifurcation diagrams for Theorem \protect\ref{T2} on (a) and for
Theorems \protect\ref{T2} $(i)$ and \protect\ref{T3} on (b).}
\end{figure}

The results of Theorems \ref{T2} and \ref{T3} are illustrated in Figure 2.
For $a<a_{0}(p)$ assumed in $(a)$ and $(b)$, the turning from the middle to
the upper solution branch occurs in the region $\lambda<0$ demonstrating the
result of Theorem \ref{T2} $(i)$ permitting two positive solutions whenever $%
0<\lambda<\lambda _{1}(f_{\Omega })$. In $(a)$, the upper solution branch
extending passes the bifurcation point $\lambda _{1}(f_{\Omega })$ and the
lower solution branch also continues its extension to the right before
turning at a point $\lambda _{1}(f_{\Omega })+\overline{\delta }_{a}$,
giving two positive solutions whenever $\lambda =\lambda _{1}(f_{\Omega })$
described by Theorem \ref{T2} $(ii-1)$ and three positive solutions whenever
$\lambda _{1}(f_{\Omega })<\lambda <\lambda _{1}(f_{\Omega })+\overline{%
\delta }_{a}$ by Theorem \ref{T2} $(ii-2)$. For $a\geq a_{0}(p)$ assumed in $%
(b)$, the turning from the middle to the upper solution branch occurs at $%
\lambda _{a}^{+}$, thus describing the case of Theorem \ref{T3} for which
two positive solutions are found for $\lambda _{a}^{+}<\lambda<\lambda
_{1}(f_{\Omega })$. Moreover, note that $\lambda _{a}^{+}$ approaches $%
\lambda _{1}(f_{\Omega })$ from the left with increasing values of $a$.
Consequently, the turning point can be seen to edge closer to the
bifurcation point $\lambda _{1}(f_{\Omega })$ as a result of increasing $a$.

The structure of this paper is as follows. After briefly introducing some
technical lemmas in Section 2, we dicuss the mountain pass geometry of the
energy functional in Section 3. We demonstrate proofs of Theorem \ref{T1} in
Section 4 and of Theorems \ref{T2} and \ref{T3} in Section 5, respectively.

\section{Preliminaries}

We denote the following notations which will be used in the paper. Denote by
$\Vert \cdot \Vert _{r}$ the $L^{r}(\mathbb{R}^{N})$-norm for $1\leq r\leq
\infty $. A strong convergence is indicated using $"\rightarrow "$ whereas
the weak convergence $"\rightharpoonup "$. We use $o(1)$ to denote a
quantity that goes to zero as $n\rightarrow \infty $. If we take a
subsequence of a sequence $\{u_{n}\}$ we shall again denote it by $\{u_{n}\}$%
. Let $S$ be the best Sobolev constant for the embedding of $D^{1,2}(\mathbb{%
R}^{N})$ in $L^{2^{\ast }}(\mathbb{R}^{N})$, where $D^{1,2}(\mathbb{R}^{N})$
is the completion of $C_{0}^{\infty }(\mathbb{R}^{N})$ with respect to the
norm $\Vert u\Vert _{D^{1,2}}^{2}=\int_{\mathbb{R}^{N}}|\nabla u|^{2}dx$.

Let
\begin{equation*}
X=\left\{ u\in H^{1}(\mathbb{R}^{N})\ |\ \int_{\mathbb{R}^{N}}V(x)u^{2}dx<%
\infty \right\}
\end{equation*}%
be equipped with the following inner product and norm%
\begin{equation*}
\langle u,v\rangle _{\mu }=\int_{\mathbb{R}^{N}}(\nabla u\nabla v+\mu
V(x)uv)dx\quad \text{and}\quad\Vert u\Vert _{\mu }=\langle u,u\rangle _{\mu
}^{1/2}
\end{equation*}%
for $\mu >0$. By condition $(V1)$, we have
\begin{align*}
\int_{\mathbb{R}^{N}}u^{2}dx& =\int_{\{V\geq
c_{0}\}}u^{2}dx+\int_{\{V<c_{0}\}}u^{2}dx \\
& \leq \frac{1}{\mu c_{0}}\int_{\mathbb{R}^{N}}\mu V(x)u^{2}dx+\frac{%
|\{V<c_{0}\}|^{\frac{2^{\ast }-2}{2^{\ast }}}}{S^{2}}\Vert u\Vert
_{D^{1,2}}^{2},
\end{align*}%
which implies that the embedding $X\hookrightarrow H^{1}(\mathbb{R}^{N})$ is
continuous. Furthermore, for all $2\leq r\leq 2^{\ast }$, it holds%
\begin{align}
\int_{\mathbb{R}^{N}}|u|^{r}dx& \leq \left( \int_{\mathbb{R}%
^{N}}u^{2}dx\right) ^{\frac{2^{\ast }-r}{2^{\ast }-2}}\left( \int_{\mathbb{R}%
^{N}}|u|^{2^{\ast }}dx\right) ^{\frac{r-2}{2^{\ast }-2}}  \notag \\
& \leq \left( \frac{1}{\mu c_{0}}\int_{\mathbb{R}^{N}}\mu V(x)u^{2}dx+\frac{%
|\{V<c_{0}\}|^{\frac{2^{\ast }-2}{2^{\ast }}}}{S^{2}}\Vert u\Vert
_{D^{1,2}}^{2}\right) ^{\frac{2^{\ast }-r}{2^{\ast }-2}}\left( \frac{\Vert
u\Vert _{D^{1,2}}^{2^{\ast }}}{S^{2^{\ast }}}\right) ^{\frac{r-2}{2^{\ast }-2%
}}  \label{r1} \\
& \leq |\{V<c_{0}\}|^{\frac{2^{\ast }-r}{2^{\ast }}}S^{-r}\Vert u\Vert _{\mu
}^{r}  \label{r2}
\end{align}%
for all $\mu \geq \mu _{0}:=S^{2}\left( c_{0}|\{V<c_{0}\}|^{\frac{2^{\ast }-2%
}{2^{\ast }}}\right) ^{-1}$.

Define the energy functional $J_{a,\lambda }^{\mu }:X\rightarrow \mathbb{R}$
by
\begin{equation*}
J_{a,\lambda }^{\mu }(u)=\frac{a}{4}\Vert u\Vert _{D^{1,2}}^{4}+\frac{1}{2}%
\Vert u\Vert _{\mu }^{2}-\frac{\lambda }{2}\int_{\mathbb{R}^{N}}f(x)u^{2}dx-%
\frac{1}{p}\int_{\mathbb{R}^{N}}g(x)|u|^{p}dx.
\end{equation*}%
$J_{a,\lambda }^{\mu }$ is a $C^{1}$ functional with the derivative given by
\begin{eqnarray*}
\langle (J_{a,\lambda }^{\mu })^{\prime }(u),\varphi \rangle &=&a\Vert
u\Vert _{D^{1,2}}^{2}\int_{\mathbb{R}^{N}}\nabla u\nabla \varphi dx+\int_{%
\mathbb{R}^{N}}(\nabla u\nabla \varphi +\mu V(x)u\varphi )dx \\
&&-\lambda \int_{\mathbb{R}^{N}}f(x)u\varphi dx-\int_{\mathbb{R}%
^{N}}g(x)|u|^{p-2}u\varphi dx
\end{eqnarray*}%
for all $\varphi \in X$, where $(J_{a,\lambda }^{\mu })^{\prime }$ denotes
the Fr\'{e}chet derivative of $J_{a,\lambda }^{\mu }.$ One can see that the
critical points of $J_{a,\lambda }^{\mu }$ are corresponding to the
solutions of Eq. $(K_{a,\lambda }^{\mu }).$

In what follows we consider the following eigenvalue problem:
\begin{equation}
-\Delta u+\mu V(x)u=\lambda f(x)u\quad \text{in}\quad X.  \label{eb-mu}
\end{equation}%
In order to find the positive principal eigenvalue of Eq. (\ref{eb-mu}), we
need to solve the following minimization problem:
\begin{equation*}
\min \left\{ \int_{\mathbb{R}^{N}}\left( |\nabla u|^{2}+\mu V(x)u^{2}\right)
dx\ |\ u\in X\text{ and }\int_{\mathbb{R}^{N}}f(x)u^{2}dx=1\right\} .
\end{equation*}%
Denote
\begin{equation}
\lambda _{1,\mu }(f)=\inf \left\{ \int_{\mathbb{R}^{N}}\left( |\nabla
u|^{2}+\mu V(x)u^{2}\right) dx\ |\ u\in X\text{ and }\int_{\mathbb{R}%
^{N}}f(x)u^{2}dx=1\right\} .  \label{1mu}
\end{equation}%
Using condition $(D1)$ and H\"{o}lder inequality gives
\begin{equation*}
\frac{\int_{\mathbb{R}^{N}}\left( |\nabla u|^{2}+\mu V(x)u^{2}\right) dx}{%
\int_{\mathbb{R}^{N}}f(x)u^{2}dx}\geq \frac{\Vert u\Vert _{D^{1,2}}^{2}}{%
\Vert f\Vert _{N/2}S^{-2}\Vert u\Vert _{D^{1,2}}^{2}}>0,
\end{equation*}%
which implies that $\lambda _{1,\mu }(f)\geq S^{2}\Vert f\Vert _{N/2}^{-1}>0$%
. Moreover, by condition $(V2)$ one has%
\begin{equation*}
\inf_{u\in X\backslash \{0\}}\frac{\int_{\mathbb{R}^{N}}\left( |\nabla
u|^{2}+\mu V(x)u^{2}\right) dx}{\int_{\mathbb{R}^{N}}f(x)u^{2}dx}\leq
\inf_{u\in H_{0}^{1}(\Omega )\backslash \{0\}}\frac{\int_{\Omega }|\nabla
u|^{2}dx}{\int_{\Omega }f_{\Omega }(x)u^{2}dx},
\end{equation*}%
which indicates that $\lambda _{1,\mu }(f)\leq \lambda _{1}(f_{\Omega })$
for all $\mu >0$. Then the following result is proved.

\begin{lemma}
(\cite[Lemma 3.2]{ZLW}) \label{L-eb-1} For each $\mu >0$ there exists a
positive function $\phi _{1,\mu }\in X$ with $\int_{\mathbb{R}^{N}}f(x)\phi
_{1,\mu }^{2}dx=1$ such that
\begin{equation*}
\lambda _{1,\mu }(f)=\int_{\mathbb{R}^{N}}\left( |\nabla \phi _{1,\mu
}|^{2}+\mu V(x)\phi _{1,\mu }^{2}\right) dx<\lambda _{1}(f_{\Omega }).
\end{equation*}%
Furthermore, it holds%
\begin{equation}
\lambda _{1,\mu }(f)\rightarrow \lambda _{1}^{-}(f_{\Omega })\quad \text{and}%
\quad \phi _{1,\mu }\rightarrow \phi _{1}\ \text{in}\ X\ \text{as }\mu
\rightarrow \infty .  \label{lam1}
\end{equation}
\end{lemma}

Note that we can find the other positive eigenvalues of Eq. (\ref{eb-mu}) by
solving the following problem:
\begin{equation}
\min \left\{ \int_{\mathbb{R}^{N}}\left( |\nabla u|^{2}+\mu V(x)u^{2}\right)
dx\ |\ u\in X,\ \int_{\mathbb{R}^{N}}f(x)u^{2}dx=1\ \text{and }\langle
u,\phi _{1,\mu }\rangle _{\mu }=0\right\} .  \label{m2}
\end{equation}%
Denote
\begin{equation*}
\lambda _{2,\mu }(f)=\inf \left\{ \int_{\mathbb{R}^{N}}\left( |\nabla
u|^{2}+\mu V(x)u^{2}\right) dx\ |\ u\in X,\ \int_{\mathbb{R}%
^{N}}f(x)u^{2}dx=1\ \text{and }\langle u,\phi _{1,\mu }\rangle _{\mu
}=0\right\} .
\end{equation*}%
In order to solve problem $(\ref{m2}),$ we need some known lemmas as follows.

\begin{lemma}
\label{L2.13} (\cite[Lemma 2.13]{W}) If $N\geq 3$ and $f\in L^{N/2}(\mathbb{R%
}^{N})$, the functional $u\mapsto \int_{\mathbb{R}^{N}}f(x)u^{2}dx$ is
weakly continuous on $H^{1}(\mathbb{R}^{N}).$
\end{lemma}

\begin{lemma}
\label{L3.1}(\cite[Lemma 3.1]{ZLW}) Let $\mu _{n}\rightarrow \infty $ as $%
n\rightarrow \infty $ and $\{v_{n}\}\subset X$ with $\Vert v_{n}\Vert _{\mu
_{n}}\leq C_{0}$ for some $C_{0}>0$. Then there exist a subsequence $%
\{v_{n}\}$ and $v_{0}\in H_{0}^{1}(\Omega )$ such that $v_{n}\rightharpoonup
v_{0}$ in $X$ and $v_{n}\rightarrow v_{0}$ in $L^{r}(\mathbb{R}^{N})$ for
all $2\leq r<2^{\ast }.$
\end{lemma}

Then we have the following result.

\begin{lemma}
\label{L-eb-2}For each $\mu >0$ there exists a function $\phi _{2,\mu }\in X$
with $\int_{\mathbb{R}^{N}}f(x)\phi _{2,\mu }^{2}dx=1$ and $\langle \phi
_{2,\mu },\phi _{1,\mu }\rangle _{\mu }=0$ such that
\begin{equation*}
\lambda _{2,\mu }(f)=\int_{\mathbb{R}^{N}}\left( |\nabla \phi _{2,\mu
}|^{2}+\mu V(x)\phi _{2,\mu }^{2}\right) dx.
\end{equation*}%
Furthermore, it holds
\begin{equation}
\frac{\lambda _{1}(f_{\Omega })+\lambda _{2}(f_{\Omega })}{2}<\lambda
_{2,\mu }(f)\quad \text{for $\mu $ sufficiently large.}  \label{lam2}
\end{equation}
\end{lemma}

\begin{proof}
Let $\{u_{n}\}\subset X$ be a minimizing sequence of problem $(\ref{m2})$.
Clearly, it is bounded. Then there exist a subsequence $\{u_{n}\}$ and $\phi
_{2,\mu }\in X$ such that $u_{n}\rightharpoonup \phi _{2,\mu }$ in $X$,
which implies that $\langle \phi _{2,\mu },\phi _{1,\mu }\rangle _{\mu }=0$.
By Lemma \ref{L2.13} and the fact of $X\hookrightarrow H^{1}(\mathbb{R}^{N})$%
, we have
\begin{equation*}
\int_{\mathbb{R}^{N}}f(x)\phi _{2,\mu }^{2}dx=\lim_{n\rightarrow \infty
}\int_{\mathbb{R}^{N}}f(x)u_{n}^{2}dx=1.
\end{equation*}%
We now prove that $u_{n}\rightarrow \phi _{2,\mu }$ in $X$. If it is false,
then
\begin{equation*}
\int_{\mathbb{R}^{N}}(|\nabla \phi _{2,\mu }|^{2}+\mu V(x)\phi _{2,\mu
}^{2})dx<\liminf_{n\rightarrow \infty }\Vert u_{n}\Vert _{\mu }^{2}=\lambda
_{2,\mu }(f),
\end{equation*}%
which is impossible due to the definition of $\lambda _{2,\mu }(f)$. So $%
u_{n}\rightarrow \phi _{2,\mu }$ in $X$ and $\lambda _{2,\mu }(f)=\Vert \phi
_{2,\mu }\Vert _{\mu }^{2}.$

Next, we show that $\frac{1}{2}(\lambda _{1}(f_{\Omega })+\lambda
_{2}(f_{\Omega }))<\lambda _{2,\mu }(f)$ for all $\mu $ sufficiently large.
Suppose on the contrary. Then there exists a sequence $\{\lambda _{2,\mu
_{n}}(f)\}$ such that
\begin{equation*}
\lambda _{2,\mu _{n}}(f)\leq \frac{1}{2}(\lambda _{1}(f_{\Omega })+\lambda
_{2}(f_{\Omega }))\text{ as }n\rightarrow \infty.
\end{equation*}%
Let $v_{n}=\phi _{2,\mu _{n}}$ be the minimizer of $\lambda _{2,\mu
_{n}}(f). $ Then it holds $\int_{\mathbb{R}^{N}}f(x)v_{n}^{2}dx=1,\langle
v_{n},\phi _{1,\mu _{n}}\rangle _{\mu _{n}}=0$ and
\begin{equation}
\Vert v_{n}\Vert _{\mu _{n}}^{2}=\lambda _{2,\mu _{n}}(f)\leq \frac{1}{2}%
(\lambda _{1}(f_{\Omega })+\lambda _{2}(f_{\Omega })).  \label{m-2}
\end{equation}%
By $(\ref{m-2})$ and Lemma \ref{L3.1}, there exist a subsequence $\{v_{n}\}$
and $v_{0}\in H_{0}^{1}(\Omega )$ such that $v_{n}\rightharpoonup v_{0}$ in $%
X$ and $v_{n}\rightarrow v_{0}$ in $L^{r}(\mathbb{R}^{N})$ for all $2\leq
r<2^{\ast }.$ Then we have
\begin{equation}
\int_{\Omega }f_{\Omega }v_{0}^{2}dx=\lim_{n\rightarrow \infty }\int_{%
\mathbb{R}^{N}}f(x)v_{n}^{2}dx=1  \label{m-3}
\end{equation}%
and
\begin{equation}
\int_{\Omega }|\nabla v_{0}|^{2}dx=\int_{\mathbb{R}^{N}}\left( |\nabla
v_{0}|^{2}+V(x)v_{0}^{2}\right) dx\leq \liminf_{n\rightarrow \infty }\Vert
v_{n}\Vert _{\mu _{n}}^{2}\leq \frac{1}{2}(\lambda _{1}(f_{\Omega })+\lambda
_{2}(f_{\Omega })).  \label{m-4}
\end{equation}%
According to $(\ref{lam1})$, we deduce that
\begin{equation}
\Vert \phi _{1,\mu _{n}}-\phi _{1}\Vert _{\mu _{n}}\rightarrow 0\quad \text{%
as}\ n\rightarrow \infty .  \label{m-5}
\end{equation}%
It follows from $(\ref{m-2}),(\ref{m-5})$ and $v_{n}\rightharpoonup v_{0}$
in $X$ that
\begin{equation}
\int_{\Omega }\nabla v_{0}\nabla \phi _{1}dx=\lim_{n\rightarrow \infty
}\langle v_{n},\phi _{1,\mu _{n}}\rangle _{\mu _{n}}=0.  \label{m-6}
\end{equation}%
From $(\ref{m-3})$ and $(\ref{m-6})$, we conclude that $\int_{\Omega
}|\nabla v_{0}|^{2}dx\geq \lambda _{2}(f_{\Omega })$, which is a
contradiction with $(\ref{m-4})$. Consequently, this completes the proof.
\end{proof}

\section{Mountain pass geometry}

Let us start this section by recalling the well-known the mountain pass
theorem \cite{AR} as follows.

\begin{theorem}
\label{T4}Let $E$ be a Banach space, $J\in C^{1}(E,\mathbb{R}),v\in E$ and $%
\rho >0$ be such that $\Vert v\Vert >\rho $ and
\begin{equation*}
b:=\inf_{\Vert u\Vert =\rho }J(u)>J(0)\geq J(v).
\end{equation*}%
If $J$ satisfies the Palais-Smale condition at level $\alpha :=\inf_{\gamma
\in \Gamma }\max_{t\in \lbrack 0,1]}J(\gamma (t))$ with
\begin{equation*}
\Gamma :=\{\gamma \in C([0,1],E)\ |\ \gamma (0)=0,\gamma (1)=v\},
\end{equation*}%
then $\alpha $ is a critical value of $J$ and $\alpha \geq b$.
\end{theorem}

We say that the functional $J_{a,\lambda }^{\mu }$ satisfies Palais-Smale
condition at level $\alpha \in \mathbb{R}$ ($(PS)_{\alpha }$-condition for
short) if any sequence $\{u_{n}\}\subset X$ with $J_{a,\lambda }^{\mu
}(u_{n})\rightarrow \alpha $ and $(J_{a,\lambda }^{\mu })^{\prime
}(u_{n})\rightarrow 0$ has a convergent subsequence. Such sequence is called
a Palais-Smale sequence at $\alpha $ ($(PS)_{\alpha }$-sequence for short).

In what follows, we prove that the functional $J_{a,\lambda }^{\mu }$
satisfies the mountain pass geometry.

\begin{lemma}
\label{L1} Suppose that $N=3,4<p<6$ and conditions $(V1)-(V2),(D1)-(D2)$
hold. Then for each $a>0,$ there exists a number $\delta _{a}>0$ such that
for each $0<\lambda <\lambda _{1}(f_{\Omega })+\delta _{a},$ there exist $%
\rho _{a,\lambda }>0$ and $e_{0}\in H_{0}^{1}(\Omega )$ such that%
\begin{equation*}
\Vert e_{0}\Vert _{\mu }>\rho _{a,\lambda }\text{ and }\inf_{\Vert u\Vert
_{\mu }=\rho _{a,\lambda }}J_{a,\lambda }^{\mu }(u)>0>J_{a,\lambda }^{\mu
}(e_{0})
\end{equation*}%
for $\mu $ sufficiently large.
\end{lemma}

\begin{proof}
First of all, we show that there exists a number $\delta _{a}>0$ such that
for each $0<\lambda <\lambda _{1}(f_{\Omega })+\delta _{a},$ there exists a
number $\rho _{a,\lambda }>0$ such that $\inf_{\Vert u\Vert _{\mu }=\rho
_{a,\lambda }}J_{a,\lambda }^{\mu }(u)>0$ for $\mu $ sufficiently large. Now
we need to separate the proof in two cases as follows.

Case $(i):0<\lambda <\lambda _{1}(f_{\Omega })$. It follows from $(\ref{lam1}%
)$ that
\begin{equation*}
\lambda _{1,\mu }(f)\geq \frac{\lambda _{1}(f_{\Omega })+\lambda }{2}\quad
\text{for $\mu $ sufficiently large,}
\end{equation*}%
which indicates that%
\begin{equation}
\frac{1}{2}\left( 1-\frac{\lambda }{\lambda _{1,\mu }(f)}\right) \geq \frac{1%
}{2}\left( \frac{\lambda _{1}(f_{\Omega })-\lambda }{\lambda _{1}(f_{\Omega
})+\lambda }\right) \quad \text{for $\mu $ sufficiently large.}
\label{lam:1-mu}
\end{equation}%
By $(\ref{r2}),(\ref{1mu})$ and $(\ref{lam:1-mu})$, one has
\begin{equation*}
J_{a,\lambda }^{\mu }(u)\geq \frac{a}{4}\Vert u\Vert _{D^{1,2}}^{4}+\frac{1}{%
2}\left( \frac{\lambda _{1}(f_{\Omega })-\lambda }{\lambda _{1}(f_{\Omega
})+\lambda }\right) \Vert u\Vert _{\mu }^{2}-\frac{\Vert g\Vert _{\infty
}|\{V<c_{0}\}|^{(6-p)/6}}{pS^{p}}\Vert u\Vert _{\mu }^{p}
\end{equation*}%
for $\mu $ sufficiently large. Let
\begin{equation*}
\rho _{\lambda }=\left[ \frac{1}{4}\left( \frac{\lambda _{1}(f_{\Omega
})-\lambda }{\lambda _{1}(f_{\Omega })+\lambda }\right) \frac{pS^{p}}{\Vert
g\Vert _{\infty }|\{V<c_{0}\}|^{(6-p)/6}}\right] ^{1/(p-2)}>0.
\end{equation*}%
Then for all $u\in X$ with $\Vert u\Vert _{\mu }=\rho _{\lambda }$, we have
\begin{equation*}
J_{a,\lambda }^{\mu }(u)\geq \frac{1}{4}\left( \frac{\lambda _{1}(f_{\Omega
})-\lambda }{\lambda _{1}(f_{\Omega })+\lambda }\right) \rho _{\lambda
}^{2}>0,
\end{equation*}%
which indicates that $\inf_{\Vert u\Vert _{\mu }=\rho _{\lambda
}}J_{a,\lambda }^{\mu }(u)>0.$

Case $(ii):\lambda \geq \lambda _{1}(f_{\Omega })$. For each $u\in X$, by
the orthogonal decomposition theorem, there exist $t\in \mathbb{R}$ and $%
w\in X$ with $\langle w,\phi _{1,\mu }\rangle _{\mu }=0$ such that $u=t\phi
_{1,\mu }+w$. Clearly, it holds%
\begin{equation}
\Vert u\Vert _{\mu }^{2}=\lambda _{1,\mu }(f)t^{2}+\Vert w\Vert _{\mu }^{2}.
\label{12}
\end{equation}%
Moreover, we obtain
\begin{equation}
\lambda _{2,\mu }(f)\int_{\mathbb{R}^{3}}f(x)w^{2}dx\leq \Vert w\Vert _{\mu
}^{2}  \label{13}
\end{equation}%
and
\begin{equation}
\lambda _{1,\mu }(f)\int_{\mathbb{R}^{3}}f(x)\phi _{1,\mu }wdx=\int_{\mathbb{%
R}^{3}}(\nabla \phi _{1,\mu }\nabla w+\mu V(x)\phi _{1,\mu }w)dx=0.
\label{14}
\end{equation}%
For the functional $J_{a,\lambda }^{\mu }$, it follows from $(\ref{r2}),(\ref%
{12})-(\ref{14})$ that%
\begin{align}
J_{a,\lambda }^{\mu }(u)& =\frac{a}{4}\Vert u\Vert _{D^{1,2}}^{4}+\frac{1}{2}%
(\lambda _{1,\mu }(f)t^{2}+\Vert w\Vert _{\mu }^{2})  \notag \\
& \text{ \ }-\frac{\lambda }{2}\int_{\mathbb{R}^{3}}(t^{2}f(x)\phi _{1,\mu
}^{2}+2tf(x)\phi _{1,\mu }w+f(x)w^{2})dx-\frac{1}{p}\int_{\mathbb{R}%
^{3}}g(x)|u|^{p}dx  \notag \\
& \geq \frac{a}{4}\Vert u\Vert _{D^{1,2}}^{4}+\frac{1}{2}\left( 1-\frac{%
\lambda }{\lambda _{1,\mu }(f)}\right) \lambda _{1,\mu }(f)t^{2}+\frac{1}{2}%
\left( 1-\frac{\lambda }{\lambda _{2,\mu }(f)}\right) \Vert w\Vert _{\mu
}^{2}  \notag \\
& \text{ \ }-\frac{\Vert g\Vert _{\infty }|\{V<c_{0}\}|^{(6-p)/6}}{pS^{p}}%
\Vert u\Vert _{\mu }^{p}  \notag \\
& \geq \frac{a}{4}\Vert u\Vert _{D^{1,2}}^{4}-|\theta _{1,\mu }|\Vert u\Vert
_{\mu }^{2}+\left( \theta _{2,\mu }-\theta _{1,\mu }\right) \Vert w\Vert
_{\mu }^{2}-\frac{\Vert g\Vert _{\infty }|\{V<c_{0}\}|^{(6-p)/6}}{pS^{p}}%
\Vert u\Vert _{\mu }^{p},  \label{15}
\end{align}%
where
\begin{equation}
\theta _{1,\mu }:=\frac{1}{2}\left( 1-\frac{\lambda }{\lambda _{1,\mu }(f)}%
\right) \quad \text{and}\quad \theta _{2,\mu }:=\frac{1}{2}\left( 1-\frac{%
\lambda }{\lambda _{2,\mu }(f)}\right) .  \label{1-2}
\end{equation}%
Since $\lambda _{1,\mu }(f)<\lambda _{1}(f_{\Omega })$, by $(\ref{lam2})$
one has
\begin{equation}
\theta _{2,\mu }-\theta _{1,\mu }\geq \frac{1}{2}\left( 1-\frac{\lambda
_{1,\mu }(f)}{\lambda _{2,\mu }(f)}\right) \geq \frac{\lambda _{2}(f_{\Omega
})-\lambda _{1}(f_{\Omega })}{2(\lambda _{2}(f_{\Omega })+\lambda
_{1}(f_{\Omega }))}=:\Lambda _{0}  \label{17}
\end{equation}%
for $\mu $ sufficiently large. Moreover, since $\phi _{1,\mu }\rightarrow
\phi _{1}$ in $X$ as $\mu\rightarrow\infty$, we conclude that $\phi _{1,\mu }\rightarrow \phi _{1}$ in
$D^{1,2}(\mathbb{R}^{3})$ as $\mu\rightarrow\infty$, which implies that%
\begin{equation}
\Vert \phi _{1,\mu }\Vert _{D^{1,2}}^{4}\geq \frac{1}{2}\lambda
_{1}^{2}(f_{\Omega })\quad \text{for $\mu $ sufficiently large.}  \label{18}
\end{equation}%
For the non-local term, we deduce that
\begin{equation}
\left\vert 4t^{3}\Vert \phi _{1,\mu }\Vert _{D^{1,2}}^{2}\int_{\mathbb{R}%
^{3}}\nabla \phi _{1,\mu }\nabla wdx\right\vert \leq \frac{3}{4}t^{4}\Vert
\phi _{1,\mu }\Vert _{D^{1,2}}^{4}+4t^{2}\left( \int_{\mathbb{R}^{3}}\nabla
\phi _{1,\mu }\nabla wdx\right) ^{2}+16\Vert w\Vert _{D^{1,2}}^{4}
\label{16}
\end{equation}%
and
\begin{equation}
\left\vert 4t\Vert w\Vert _{D^{1,2}}^{2}\int_{\mathbb{R}^{3}}\nabla \phi
_{1,\mu }\nabla wdx\right\vert \leq 2t^{2}\Vert \phi _{1,\mu }\Vert
_{D^{1,2}}^{2}\Vert w\Vert _{D^{1,2}}^{2}+2\Vert w\Vert _{D^{1,2}}^{4}.
\label{20}
\end{equation}%
Then from $(\ref{18})-(\ref{20})$ it follows that%
\begin{align}
\Vert u\Vert _{D^{1,2}}^{4}& =\Vert t\phi _{1,\mu }+w\Vert
_{D^{1,2}}^{4}\geq \frac{\Vert \phi _{1,\mu }\Vert _{D^{1,2}}^{4}}{4\lambda
_{1,\mu }^{2}(f)}\left( \Vert u\Vert _{\mu }^{2}-\Vert w\Vert _{\mu
}^{2}\right) ^{2}-17\Vert w\Vert _{D^{1,2}}^{4}  \notag \\
& \geq \frac{1}{16}\Vert u\Vert _{\mu }^{4}-\frac{137}{8}\Vert w\Vert _{\mu
}^{4}.  \label{21}
\end{align}%
Combining $(\ref{15}),(\ref{17})$ with $(\ref{21})$, for $\mu $ sufficiently
large one has
\begin{align*}
J_{a,\lambda }^{\mu }(u)& \geq \frac{a}{64}\Vert u\Vert _{\mu }^{4}-\frac{137%
}{32}a\Vert w\Vert _{\mu }^{4}-|\theta _{1,\mu }|\Vert u\Vert _{\mu
}^{2}+\Lambda _{0}\Vert w\Vert _{\mu }^{2}-\frac{\Vert g\Vert _{\infty
}|\{V<c_{0}\}|^{(6-p)/6}}{pS^{p}}\Vert u\Vert _{\mu }^{p} \\
& =-|\theta _{1,\mu }|\Vert u\Vert _{\mu }^{2}+\Vert u\Vert _{\mu
}^{4}\left( \frac{a}{64}-\frac{\Vert g\Vert _{\infty }|\{V<c_{0}\}|^{(6-p)/6}%
}{pS^{p}}\Vert u\Vert _{\mu }^{p-4}\right) +\Vert w\Vert _{\mu }^{2}\left(
\Lambda _{0}-\frac{137}{32}a\Vert w\Vert _{\mu }^{2}\right) .
\end{align*}%
This implies that there exists a number%
\begin{equation}
\rho _{a}=\min \left\{ \left( \frac{apS^{p}}{128\Vert g\Vert _{\infty
}|\{V<c_{0}\}|^{(6-p)/6}}\right) ^{1/(p-4)},\left( \frac{32\Lambda _{0}}{137a%
}\right) ^{1/2}\right\}   \label{22}
\end{equation}%
such that for all $u\in X$ with $\Vert u\Vert _{\mu }=\rho _{a},$%
\begin{equation*}
J_{a,\lambda }^{\mu }(u)\geq -|\theta _{1,\mu }|\rho _{a}^{2}+\frac{a}{128}%
\rho _{a}^{4}.
\end{equation*}%
Thus, we deduce that
\begin{equation*}
J_{a,\lambda }^{\mu }(u)\geq \frac{a}{256}\rho _{a}^{4}>0
\end{equation*}%
for each $\lambda _{1,\mu }(f)\leq \lambda <\lambda _{1,\mu }(f)+\delta
_{a,\mu }$, where
\begin{equation}
\delta _{a,\mu }:=\frac{\lambda _{1,\mu }(f)}{128}a\rho _{a}^{2}.  \label{23}
\end{equation}%
So, according to Case $(i)-(ii),$ for each $a>0$ and $0<\lambda <\lambda
_{1,\mu }(f)+\delta _{a,\mu }$, we have
\begin{equation*}
\inf_{\Vert u\Vert _{\mu }=\rho _{a,\lambda }}J_{a,\lambda }^{\mu }(u)>0%
\text{ for }\mu \text{ sufficiently large,}
\end{equation*}%
where%
\begin{equation*}
\rho _{a,\lambda }:=\left\{
\begin{array}{ll}
\rho _{\lambda } & \text{ for }0<\lambda <\lambda _{1}(f_{\Omega }), \\
\rho _{a} & \text{ for }\lambda _{1}(f_{\Omega })\leq \lambda <\lambda
_{1}(f_{\Omega })+\delta _{a,\mu }.%
\end{array}%
\right.
\end{equation*}%
Set
\begin{equation*}
\delta _{a}=\min \left\{ a^{\frac{p-2}{p-4}}C_{1},C_{2}\right\} >0,
\end{equation*}%
where%
\begin{equation*}
C_{1}:=\frac{\lambda _{1}(f_{\Omega })}{4}\left( \frac{1}{128}\right) ^{%
\frac{p-2}{p-4}}\left( \frac{pS^{p}}{\Vert g\Vert _{\infty
}|\{V<c_{0}\}|^{(6-p)/6}}\right) ^{2/(p-4)}
\end{equation*}%
and%
\begin{equation*}
C_{2}:=\frac{\lambda _{1}(f_{\Omega })(\lambda _{2}(f_{\Omega })-\lambda
_{1}(f_{\Omega }))}{4384(\lambda _{2}(f_{\Omega })+\lambda _{1}(f_{\Omega }))%
}.
\end{equation*}%
Then by $(\ref{22})-(\ref{23}),$ we obtain that for $\mu $ sufficiently
large,%
\begin{equation*}
\lambda _{1}(f_{\Omega })+\delta _{a}\leq \lambda _{1,\mu }(f)+2\delta
_{a}\leq \lambda _{1,\mu }(f)+\delta _{a,\mu }.
\end{equation*}%
Hence, for each $a>0$ and $0<\lambda <\lambda _{1}(f_{\Omega })+\delta _{a}$
it holds
\begin{equation*}
\inf_{\Vert u\Vert _{\mu }=\rho _{a,\lambda }}J_{a,\lambda }^{\mu }(u)>0%
\text{ for }\mu \text{ sufficiently large.}
\end{equation*}%
Next, we show that there exists $e_{0}\in H_{0}^{1}(\Omega )$ such that $%
\Vert e_{0}\Vert _{\mu }>\rho _{a,\lambda }$ and $J_{a,\lambda }^{\mu
}(e_{0})<0$. Owing to condition $(D2)$, we can take $\varphi \in
H_{0}^{1}(\Omega )$ such that $\int_{\mathbb{R}^{3}}g(x)|\varphi |^{p}dx>0$.
Then for any $t>0$, we have%
\begin{equation*}
J_{a,\lambda }^{\mu }(t\varphi )=\frac{1}{2}\left( \Vert \varphi \Vert _{\mu
}^{2}-\lambda \int_{\mathbb{R}^{3}}f(x)\varphi ^{2}dx\right) t^{2}+\frac{a}{4%
}\Vert \varphi \Vert _{D^{1,2}}^{4}t^{4}-\frac{\int_{\mathbb{R}%
^{3}}g(x)|\varphi |^{p}dx}{p}t^{p}.
\end{equation*}%
This implies that there exists $t_{0}>0$ such that $\Vert t_{0}\varphi \Vert
_{\mu }>\rho _{a,\lambda }$ and $J_{a,\lambda }^{\mu }(t_{0}\varphi )<0$.
Consequently, we complete the proof.
\end{proof}

\begin{lemma}
\label{L2} Suppose that $N\geq 3,2<p<\min \{4,2^{\ast }\}$ and conditions $%
(V1)-(V2),(D1)-(D2)$ hold. Then for each $0<a<a_{0}(p),$ we have the
following results.\newline
$(i)$ For each $0<\lambda <\lambda _{1}(f_{\Omega }),$ there exist a number $%
\overline{\rho }_{a,\lambda }>0$ and $e_{0}\in H_{0}^{1}(\Omega )$ such that
\begin{equation}
\Vert e_{0}\Vert _{\mu }>\overline{\rho }_{a,\lambda }\text{ and }%
\inf_{\Vert u\Vert _{\mu }=\overline{\rho }_{a,\lambda }}J_{a,\lambda }^{\mu
}(u)>0>J_{a,\lambda }^{\mu }(e_{0})  \label{24}
\end{equation}%
for $\mu $ sufficiently large.\newline
$(ii)$ If $\int_{\Omega }g(x)\phi _{1}^{p}dx<0,$ then there exists a number $%
\overline{\delta }_{a}>0$ such that for each $\lambda _{1}(f_{\Omega })\leq
\lambda <\lambda _{1}(f_{\Omega })+\overline{\delta }_{a},$ there exist $%
\overline{\rho }_{a,\lambda }>0$ and $e_{0}\in H_{0}^{1}(\Omega )$ such that
$(\ref{24})$ holds for $\mu $ sufficiently large.
\end{lemma}

\begin{proof}
$(i)$ It follows from $(\ref{r2}),(\ref{1mu})$ and $(\ref{lam:1-mu})$ that
\begin{eqnarray}
J_{a,\lambda }^{\mu }(u) &\geq &\frac{a}{4}\Vert u\Vert _{D^{1,2}}^{4}+\frac{%
1}{2}\left( 1-\frac{\lambda }{\lambda _{1,\mu }(f)}\right) \Vert u\Vert
_{\mu }^{2}-\frac{\Vert g\Vert _{\infty }|\{V<c_{0}\}|^{(2^{\ast
}-p)/2^{\ast }}}{pS^{p}}\Vert u\Vert _{\mu }^{p}  \notag \\
&\geq &\frac{a}{4}\Vert u\Vert _{D^{1,2}}^{4}+\frac{1}{2}\left( \frac{%
\lambda _{1}(f_{\Omega })-\lambda }{\lambda _{1}(f_{\Omega })+\lambda }%
\right) \Vert u\Vert _{\mu }^{2}-\frac{\Vert g\Vert _{\infty
}|\{V<c_{0}\}|^{(2^{\ast }-p)/2^{\ast }}}{pS^{p}}\Vert u\Vert _{\mu }^{p}.
\label{4-0}
\end{eqnarray}%
Let
\begin{equation}
\overline{\rho }_{a,\lambda }=\min \{\overline{\rho }_{\lambda },\overline{%
\rho }_{a}\}>0,  \label{4.0}
\end{equation}%
where
\begin{equation}
\overline{\rho }_{\lambda }:=\left[ \frac{1}{4}\left( \frac{\lambda
_{1}(f_{\Omega })-\lambda }{\lambda _{1}(f_{\Omega })+\lambda }\right) \frac{%
pS^{p}}{\Vert g\Vert _{\infty }|\{V<c_{0}\}|^{(2^{\ast }-p)/2^{\ast }}}%
\right] ^{1/(p-2)}  \label{4.9}
\end{equation}%
and
\begin{equation}
\overline{\rho }_{a}:=\left( \frac{(p-2)\Gamma _{p}}{ap}\right) ^{1/(4-p)}.
\label{4.2}
\end{equation}%
Then by $(\ref{4-0}),$ for all $u\in X$ with $\Vert u\Vert _{\mu }=\overline{%
\rho }_{a,\lambda }$ one has%
\begin{equation*}
J_{a,\lambda }^{\mu }(u)\geq \frac{1}{4}\left( \frac{\lambda _{1}(f_{\Omega
})-\lambda }{\lambda _{1}(f_{\Omega })+\lambda }\right) \overline{\rho }%
_{a,\lambda }^{2}>0.
\end{equation*}%
Since $a<a_{0}(p)$, there exists $\varphi _{a}\in H_{0}^{1}(\Omega )$ with $%
\int_{\Omega }f_{\Omega }(x)\varphi _{a}^{2}dx\geq 0$ and $\int_{\Omega
}g(x)|\varphi _{a}|^{p}dx>0$ such that
\begin{equation}
a<2(p-2)\left( 4-p\right) ^{\frac{4-p}{p-2}}\left( \frac{\int_{\Omega
}g(x)|\varphi _{a}|^{p}dx}{p\left( \int_{\Omega }|\nabla \varphi
_{a}|^{2}dx\right) ^{p/2}}\right) ^{2/(p-2)}\leq a_{0}(p).  \label{e1}
\end{equation}%
Let
\begin{equation*}
t_{a}=\left( \frac{(2p-4)\Gamma _{p}}{ap}\right) ^{1/(4-p)}\left(
\int_{\Omega }|\nabla \varphi _{a}|^{2}dx\right) ^{-1/2}.
\end{equation*}%
Then by $(\ref{4.0})$ and $(\ref{e1})$, we have $\Vert t_{a}\varphi
_{a}\Vert _{\mu }>\overline{\rho }_{a,\lambda }$ and
\begin{align*}
& J_{a,\lambda }^{\mu }(t_{a}\varphi _{a}) \\
& =\frac{t_{a}^{2}}{2}\left( \int_{\Omega }|\nabla \varphi
_{a}|^{2}dx-\lambda \int_{\Omega }f_{\Omega }\varphi _{a}^{2}dx\right) +%
\frac{a}{4}\left( \int_{\Omega }|\nabla \varphi _{a}|^{2}dx\right)
^{2}t_{a}^{4}-\frac{\int_{\Omega }g(x)|\varphi _{a}|^{p}dx}{p}t_{a}^{p} \\
& =\frac{t_{a}^{2}}{2}\left[ \int_{\Omega }|\nabla \varphi
_{a}|^{2}dx-\lambda \int_{\Omega }f_{\Omega }(x)\varphi _{a}^{2}dx-\frac{%
(4-p)\int_{\Omega }g(x)|\varphi _{a}|^{p}dx}{p}\left( \frac{%
2(p-2)\int_{\Omega }g(x)|\varphi _{a}|^{p}dx}{ap\left( \int_{\Omega }|\nabla
\varphi _{a}|^{2}dx\right) ^{2}}\right) ^{\frac{p-2}{4-p}}\right]  \\
& =\frac{t_{a}^{2}}{2}\left[ \int_{\Omega }|\nabla \varphi _{a}|^{2}dx\left(
1-\frac{a(4-p)}{2(p-2)}\left( \frac{2(p-2)\int_{\Omega }g(x)|\varphi
_{a}|^{p}dx}{ap\left( \int_{\Omega }|\nabla \varphi _{a}|^{2}dx\right) ^{p/2}%
}\right) ^{2/(4-p)}\right) -\lambda \int_{\Omega }f_{\Omega }(x)\varphi
_{a}^{2}dx\right]  \\
& <0.
\end{align*}%
$(ii)$ For each $u\in X$, by the orthogonal decomposition theorem, there
exist $t\in \mathbb{R}$ and $w\in X$ with $\langle w,\phi _{1,\mu }\rangle
_{\mu }=0$ such that $u=t\phi _{1,\mu }+w$. Using the same process in $(\ref%
{15})$ and $(\ref{17})$ gives%
\begin{eqnarray}
J_{a,\lambda }^{\mu }(u) &\geq &\frac{a}{4}\Vert u\Vert
_{D^{1,2}}^{4}-|\theta _{1,\mu }|\Vert u\Vert _{\mu }^{2}+\Lambda _{0}\Vert
w\Vert _{\mu }^{2}  \notag \\
&&-\frac{1}{p}\int_{\mathbb{R}^{N}}g(x)\left\vert t\phi _{1,\mu }\right\vert
^{p}dx-\frac{1}{p}\int_{\mathbb{R}^{N}}g(x)\left( \left\vert t\phi _{1,\mu
}+w\right\vert ^{p}-\left\vert t\phi _{1,\mu }\right\vert ^{p}\right) dx
\label{4.1}
\end{eqnarray}%
for $\mu $ sufficiently large, where $\theta _{1,\mu }$ and $\Lambda _{0}$
are as in $(\ref{1-2})$ and $(\ref{17}),$ respectively. By $(\ref{r2})$ and $%
(\ref{lam1})$, we conclude that $\int_{\mathbb{R}^{N}}g(x)\phi _{1,\mu
}^{p}dx\rightarrow \int_{\Omega }g(x)\phi _{1}^{p}dx$ as $\mu \rightarrow
\infty $, which implies that%
\begin{equation}
\int_{\mathbb{R}^{N}}g(x)\phi _{1,\mu }^{p}dx\leq \frac{1}{2}\int_{\Omega
}g(x)\phi _{1}^{p}dx<0\quad \text{for $\mu $ sufficiently large.}
\label{4.4}
\end{equation}%
By the mean value theorem, there exists $0<\theta <1$ such that%
\begin{equation}
\frac{1}{p}\int_{\mathbb{R}^{N}}g(x)(\left\vert t\phi _{1,\mu }+w\right\vert
^{p}-\left\vert t\phi _{1,\mu }\right\vert ^{p})dx=\int_{\mathbb{R}%
^{N}}g(x)\left\vert t\phi _{1,\mu }+\theta w\right\vert ^{p-2}(t\phi _{1,\mu
}+\theta w)wdx.  \label{4.5}
\end{equation}%
Using Young's inequality and $(\ref{r2})$ leads to%
\begin{eqnarray}
&&\left\vert \int_{\mathbb{R}^{N}}g(x)\left\vert t\phi _{1,\mu }+\theta
w\right\vert ^{p-2}\left( t\phi _{1,\mu }+\theta w\right) wdx\right\vert
\notag \\
&\leq &2^{p-2}\Vert g\Vert _{\infty }\int_{\mathbb{R}^{N}}(|t\phi _{1,\mu
}|^{p-1}+|\theta w|^{p-1})|w|dx  \notag \\
&\leq &2^{p-2}\Vert g\Vert _{\infty }\int_{\mathbb{R}^{N}}\left( \frac{p-1}{p%
}B^{p/(p-1)}|t\phi _{1,\mu }|^{p}+\frac{1}{pB^{p}}|w|^{p}+|w|^{p}\right) dx
\notag \\
&\leq &\frac{2^{p-2}\Vert g\Vert _{\infty }|\{V<c_{0}\}|^{\frac{2^{\ast }-p}{%
2^{\ast }}}}{pS^{p}}\left( B^{p/(p-1)}(p-1)|t|^{p}\Vert \phi _{1,\mu }\Vert
_{\mu }^{p}+\frac{(1+pB^{p})}{B^{p}}\Vert w\Vert _{\mu }^{p}\right)   \notag
\\
&\leq &\frac{\left\vert \int_{\Omega }g(x)\phi _{1}^{p}dx\right\vert }{4p}%
|t|^{p}+\frac{2^{p-2}\Vert g\Vert _{\infty }(1+pB^{p})|\{V<c_{0}\}|^{\frac{%
2^{\ast }-p}{2^{\ast }}}}{pB^{p}S^{p}}\Vert w\Vert _{\mu }^{p}\text{ for }%
\mu \text{ sufficiently large,}  \label{4.6}
\end{eqnarray}%
where%
\begin{equation*}
B:=\left( \frac{S^{p}|\int_{\Omega }g(x)\phi _{1}^{p}dx|}{2^{p}\Vert g\Vert
_{\infty }(p-1)|\{V<c_{0}\}|^{(2^{\ast }-p)/2^{\ast }}\lambda
_{1}^{p/2}(f_{\Omega })}\right) ^{(p-1)/p}.
\end{equation*}%
Thus, it follows from $(\ref{4.1})-(\ref{4.6})$ that
\begin{align}
J_{a,\lambda }^{\mu }(u)& \geq \frac{a}{4}\Vert u\Vert
_{D^{1,2}}^{4}-|\theta _{1,\mu }|\Vert u\Vert _{\mu }^{2}+\Lambda _{0}\Vert
w\Vert _{\mu }^{2}+\frac{\left\vert \int_{\Omega }g(x)\phi
_{1}^{p}dx\right\vert }{2p}|t|^{p}  \notag \\
& \text{ \ }-\frac{1}{p}\int_{\mathbb{R}^{N}}g(x)(|t\phi _{1,\mu
}+w|^{p}-|t\phi _{1,\mu }|^{p})dx  \notag \\
& \geq -|\theta _{1,\mu }|\Vert u\Vert _{\mu }^{2}+\Lambda _{0}\Vert w\Vert
_{\mu }^{2}+\frac{\left\vert \int_{\Omega }g(x)\phi _{1}^{p}dx\right\vert }{%
4p\lambda _{1}^{p/2}(f_{\Omega })}\left( \frac{\Vert u\Vert _{\mu }^{p}}{%
2^{(p-2)/2}}-\Vert w\Vert _{\mu }^{p}\right)   \notag \\
& \text{ \ }-\frac{2^{p-2}\Vert g\Vert _{\infty }(1+pB^{p})|\{V<c_{0}\}|^{%
\frac{2^{\ast }-p}{2^{\ast }}}}{pB^{p}S^{p}}\Vert w\Vert _{\mu }^{p}  \notag
\\
& \geq -|\theta _{1,\mu }|\Vert u\Vert _{\mu }^{2}+\frac{\left\vert
\int_{\Omega }g(x)\phi _{1}^{p}dx\right\vert }{2^{(p+2)/2}p\lambda
_{1}^{p/2}(f_{\Omega })}\Vert u\Vert _{\mu }^{p}  \notag \\
& \text{ \ }+\Vert w\Vert _{\mu }^{2}\left[ \Lambda _{0}-\left( \frac{%
\left\vert \int_{\Omega }g(x)\phi _{1}^{p}dx\right\vert }{4p\lambda
_{1}^{p/2}(f_{\Omega })}+\frac{2^{p-2}\Vert g\Vert _{\infty
}(1+pB^{p})|\{V<c_{0}\}|^{\frac{2^{\ast }-p}{2^{\ast }}}}{pB^{p}S^{p}}%
\right) \Vert w\Vert _{\mu }^{p-2}\right] .  \label{4.7}
\end{align}%
Let%
\begin{equation*}
\overline{\rho }_{a,\lambda }=\min \{\rho _{0},\overline{\rho }_{a}\}\text{
and }\overline{\delta }_{a,\mu }=\frac{|\int_{\Omega }g(x)\phi
_{1}^{p}dx|\lambda _{1,\mu }(f)}{2^{(p+2)/2}p\lambda _{1}^{p/2}(f_{\Omega })}%
\overline{\rho }_{a,\lambda }^{p-2},
\end{equation*}%
where%
\begin{equation*}
\rho _{0}:=\Lambda _{0}^{1/(p-2)}\left( \frac{|\int_{\Omega }g(x)\phi
_{1}^{p}dx|}{4p\lambda _{1}^{p/2}(f_{\Omega })}+\frac{2^{p-2}\Vert g\Vert
_{\infty }(1+pB^{p})|\{V<c_{0}\}|^{(2^{\ast }-p)/2^{\ast }}}{pB^{p}S^{p}}%
\right) ^{-1/(p-2)}
\end{equation*}%
and $\overline{\rho }_{a}$ is as $(\ref{4.2}).$ Then by $(\ref{4.7}),$ for
all $u\in X$ with $\Vert u\Vert _{\mu }=\overline{\rho }_{a,\lambda }$ one
has
\begin{equation}
J_{a,\lambda }^{\mu }(u)\geq -|\theta _{1,\mu }|\overline{\rho }_{a,\lambda
}^{2}+\frac{|\int_{\Omega }g(x)\phi _{1}^{p}dx|}{2^{(p+2)/2}p\lambda
_{1}^{p/2}(f_{\Omega })}\overline{\rho }_{a,\lambda }^{p}\geq \frac{%
|\int_{\Omega }g(x)\phi _{1}^{p}dx|}{2^{(p+4)/2}p\lambda
_{1}^{p/2}(f_{\Omega })}\overline{\rho }_{a,\lambda }^{p}  \label{4.8}
\end{equation}%
for each $\lambda _{1,\mu }(f)\leq \lambda <\lambda _{1,\mu }(f)+\overline{%
\delta }_{a,\mu }$. Set%
\begin{equation*}
\overline{\delta }_{a}=\frac{|\int_{\Omega }g(x)\phi _{1}^{p}dx|}{%
2^{(p+6)/2}p\lambda _{1}^{(p-2)/2}(f_{\Omega })}\overline{\rho }_{a,\lambda
}^{p-2}.
\end{equation*}%
Since $\lambda _{1,\mu }(f)\rightarrow \lambda _{1}^{-}(f_{\Omega })$ as $%
\mu \rightarrow \infty $, we conclude that
\begin{equation}
\lambda _{1}(f_{\Omega })+\overline{\delta }_{a}\leq \lambda _{1,\mu }(f)+2%
\overline{\delta }_{a}\leq \lambda _{1,\mu }(f)+\overline{\delta }_{a,\mu }.
\label{4.3}
\end{equation}%
Hence, it follows from $(\ref{4.8})$ and $(\ref{4.3})$ that for each $%
\lambda _{1}(f_{\Omega })\leq \lambda <\lambda _{1}(f_{\Omega })+\overline{%
\delta }_{a},$
\begin{equation*}
\inf_{\Vert u\Vert _{\mu }=\overline{\rho }_{a,\lambda }}J_{a,\lambda }^{\mu
}(u)>0\text{ for }\mu \text{ sufficiently large.}
\end{equation*}%
Next, repeating the same argument as in $(i)$, we can conclude that there
exists $e_{0}\in H_{0}^{1}(\Omega )$ such that $\Vert e_{0}\Vert _{\mu }>%
\overline{\rho }_{a,\lambda }$ and $J_{a,\lambda }^{\mu }(e_{0})<0$.
Consequently, we complete the proof.
\end{proof}

\begin{lemma}
\label{L3} Suppose that $N\geq 3,2<p<\min \{4,2^{\ast }\}$ and conditions $%
(V1)-(V2),(D1)-(D2)$ hold. If $\int_{\Omega }g(x)\phi _{1}^{p}dx>0$, then
for each $a\geq a_{0}(p)$ and $\lambda _{a}^{+}<\lambda <\lambda
_{1}(f_{\Omega })$, there exist $\widehat{\rho }_{a,\lambda }>0$ and $%
e_{0}\in H_{0}^{1}(\Omega )$ such that
\begin{equation*}
\Vert e_{0}\Vert _{\mu }>\widehat{\rho }_{a,\lambda }\text{ and }\inf_{\Vert
u\Vert _{\mu }=\widehat{\rho }_{a,\lambda }}J_{a,\lambda }^{\mu
}(u)>0>J_{a,\lambda }(e_{0})\text{ for }\mu \text{ sufficiently large,}
\end{equation*}
where $\lambda _{a}^{+}$ is as in $(\ref{lam:a+}).$
\end{lemma}

\begin{proof}
Let%
\begin{equation*}
\widehat{\rho }_{a,\lambda }:=\min \left\{ \overline{\rho }_{\lambda },%
\widehat{\rho }_{a}\right\} ,
\end{equation*}%
where $\overline{\rho }_{\lambda }$ is as in $(\ref{4.9})$ and%
\begin{equation*}
\widehat{\rho }_{a}:=\left( \frac{(p-2)\int_{\Omega }g(x)\phi _{1}^{p}dx}{%
ap\lambda _{1}^{p/2}(f_{\Omega })}\right) ^{1/(4-p)}.
\end{equation*}%
Then, similar to the argument in Lemma \ref{L2}$(i)$, for all $u\in X$ with $%
\Vert u\Vert _{\mu }=\widehat{\rho }_{a,\lambda }$ one has%
\begin{eqnarray*}
J_{a,\lambda }^{\mu }(u) &\geq &\frac{a}{4}\Vert u\Vert _{D^{1,2}}^{4}+\frac{%
1}{2}\left( \frac{\lambda _{1}(f_{\Omega })-\lambda }{\lambda _{1}(f_{\Omega
})+\lambda }\right) \Vert u\Vert _{\mu }^{2}-\frac{\Vert g\Vert _{\infty
}|\{V<c_{0}\}|^{(2^{\ast }-p)/2^{\ast }}}{pS^{p}}\Vert u\Vert _{\mu }^{p} \\
&\geq &\frac{1}{4}\left( \frac{\lambda _{1}(f_{\Omega })-\lambda }{\lambda
_{1}(f_{\Omega })+\lambda }\right) \widehat{\rho }_{a,\lambda }^{2}>0\text{
for }\mu \text{ sufficiently large.}
\end{eqnarray*}%
We set
\begin{equation*}
t_{a}=\left( \frac{2(p-2)\int_{\Omega }g(x)\phi _{1}^{p}dx}{ap\lambda
_{1}^{2}(f_{\Omega })}\right) ^{1/(4-p)}.
\end{equation*}%
Then it holds $\Vert t_{a}\phi _{1}\Vert _{\mu }>\widehat{\rho }_{a,\lambda
} $ and
\begin{align*}
J_{a,\lambda }(t_{a}\phi _{1})& =\frac{\lambda _{1}(f_{\Omega })-\lambda }{2}%
t_{a}^{2}+\frac{a\lambda _{1}^{2}(f_{\Omega })}{4}t_{a}^{4}-\frac{%
\int_{\Omega }g(x)\phi _{1}^{p}dx}{p}t_{a}^{p} \\
& =\frac{t_{a}^{2}}{2}(\lambda _{a}^{+}-\lambda )<0\text{ for }\lambda
>\lambda _{a}^{+}.
\end{align*}%
Consequently, we complete the proof.
\end{proof}

Finally, we state the compactness condition for the functional $J_{a,\lambda
}^{\mu },$ which has been proved in \cite{ZLW}.

\begin{lemma}
\label{PS-C}(\cite[Proposition 2.5]{ZLW}) Suppose that $N\geq 3,2<p<2^{\ast
} $ and conditions $(V1)-(V2),(D1)-(D2)$ hold. In addition, for $N=3$, we
also assume that condition $(H3)$ holds. Let $\alpha \in \mathbb{R}$ and $%
\{u_{n}\}$ be a $(PS)_{\alpha }$-sequence for $J_{a,\lambda }^{\mu }$. If
there exists $d_{0}>0$ such that $\Vert u_{n}\Vert _{\mu }<d_{0},$ then $%
\{u_{n}\}$ strongly converges in $X_{\mu }$ up to subsequence for $\mu $
sufficiently large.
\end{lemma}

\section{The Proof of Theorem \protect\ref{T1}}

\textbf{We now prove Theorem \ref{T1}:} By Lemma \ref{L1}, for each $a>0$
and $\mu $ is sufficiently large, there exists $\delta _{a}>0$ such that the
functional $J_{a,\lambda }^{\mu }$ has the mountain pass geometry whenever $%
0<\lambda <\lambda _{1}(f_{\Omega })+\delta _{a}$. Let
\begin{equation*}
\alpha _{\mu }=\inf_{\gamma \in \Gamma }\max_{0\leq s\leq 1}J_{a,\lambda
}^{\mu }(\gamma (s))\quad \text{with}\quad \Gamma =\left\{ \gamma \in
C\left( [0,1],X\right) \ |\ \gamma (0)=0,\gamma (1)=e_{0}\right\} ,
\end{equation*}%
where $e_{0}$ is as in Lemma \ref{L1}. It is clear that
\begin{equation*}
0<\alpha _{\mu }\leq \max_{0\leq s\leq 1}J_{a,\lambda }^{\mu }(se_{0})=:D_{0}
\end{equation*}%
and that $D_{0}$ is independent of $\mu $ due to $e_{0}\in H_{0}^{1}(\Omega
) $. Let $\{u_{n}\}$ be a $(PS)_{\alpha _{\mu }}$-sequence, that is $%
J_{a,\lambda }^{\mu }(u_{n})=\alpha _{\mu }+o(1)$ and $\left( J_{a,\lambda
}^{\mu }\right) ^{\prime }(u_{n})=o(1)$. In fact, since $J_{a,\lambda }^{\mu
}(u_{n})=J_{a,\lambda }^{\mu }(\left\vert u_{n}\right\vert )$ for all $n,$
we may assume that $u_{n}\geq 0.$ Then we have
\begin{align}
\alpha _{\mu }p+1& \geq pJ_{a,\lambda }^{\mu }(u_{n})-\langle \left(
J_{a,\lambda }^{\mu }\right) ^{\prime }(u_{n}),u_{n}\rangle  \notag \\
& =\frac{a(p-4)}{4}\Vert u_{n}\Vert _{D^{1,2}}^{4}+\frac{p-2}{2}\Vert
u_{n}\Vert _{\mu }^{2}-\frac{\lambda (p-2)}{2}\int_{\mathbb{R}%
^{3}}f(x)u_{n}^{2}dx.  \label{4-1}
\end{align}%
Using condition $(D1)$ and Young's inequality gives
\begin{align}
\frac{\lambda (p-2)}{2}\int_{\mathbb{R}^{3}}f(x)u_{n}^{2}dx& \leq \frac{%
\lambda (p-2)}{2}S^{-2}\Vert f\Vert _{3/2}\Vert u_{n}\Vert _{D^{1,2}}^{2}
\notag \\
& \leq \frac{a(p-4)}{4}\Vert u_{n}\Vert _{D^{1,2}}^{4}+\frac{\lambda
^{2}(p-2)^{2}\Vert f\Vert _{3/2}^{2}}{4(p-4)aS^{4}}.  \label{4-2}
\end{align}%
Combining $(\ref{4-1})$ with $(\ref{4-2})$ leads to
\begin{equation*}
D_{0}p+1\geq \alpha _{\mu }p+1\geq \frac{p-2}{2}\Vert u_{n}\Vert _{\mu }^{2}-%
\frac{\lambda ^{2}(p-2)^{2}\Vert f\Vert _{3/2}^{2}}{4(p-4)aS^{4}},
\end{equation*}%
which indicates that there exists $d_{0}>0$ such that $\Vert u_{n}\Vert
_{\mu }<d_{0}$ for $\mu $ sufficiently large. Thus, by Lemma \ref{PS-C}, the
functional $J_{a,\lambda }^{\mu }$ satisfies the $(PS)_{\alpha _{\mu }}$%
-condition. Hence, there exists $0\leq u_{0}^{\left( 1\right) }\in X$ such
that $J_{a,\lambda }^{\mu }(u_{0}^{\left( 1\right) })=\alpha _{\mu }$ and $%
(J_{a,\lambda }^{\mu })^{\prime }(u_{0}^{\left( 1\right) })=0$ for $\mu $
sufficiently large, this implies that $u_{0}^{\left( 1\right) }$ is a
nontrivial nonnegative solution of Eq. $\left( K_{a,\lambda }^{\mu }\right)
. $ The strong Maximum Principle implies that $u_{0}^{\left( 1\right) }>0$
in $\mathbb{R}^{3}.$ Therefore, the proof of part $(i)$ is completed.

To prove part $(ii)$, we consider the infimum of $J_{a,\lambda }^{\mu }$ on
the closed ball $B_{\rho _{a,\lambda }}:=\{u\in X\ |\ \Vert u\Vert _{\mu
}\leq \rho _{a,\lambda }\}$ with $\rho _{a,\lambda }$ being as in Lemma \ref%
{L1}. Note that $\rho _{a,\lambda }$ is independent of $\mu $ for $\lambda
_{1}(f_{\Omega })<\lambda <\lambda _{1}(f_{\Omega })+\delta _{a}$. Set
\begin{equation*}
\beta _{\mu }=\inf_{\Vert u\Vert _{\mu }\leq \rho _{a,\lambda }}J_{a,\lambda
}^{\mu }(u).
\end{equation*}%
Let
\begin{equation*}
J_{a,\lambda }^{\mu }(t\phi _{1})=-\frac{\lambda -\lambda _{1}(f_{\Omega })}{%
2}t^{2}+\frac{a\lambda _{1}^{2}(f_{\Omega })}{4}t^{4}-\frac{\int_{\mathbb{R}%
^{3}}g(x)\phi _{1}^{p}dx}{p}t^{p}\text{ for }t>0.
\end{equation*}%
Then for each $\lambda >\lambda _{1}(f_{\Omega })$, there exists $t_{0}>0$
such that $\Vert t_{0}\phi _{1}\Vert _{\mu }\leq \rho _{a,\lambda }$ and $%
J_{a,\lambda }^{\mu }(t_{0}\phi _{1})<0$. Moreover, we have
\begin{align*}
J_{a,\lambda }^{\mu }(u)& \geq -\frac{\lambda \Vert f\Vert _{3/2}}{2S^{2}}%
\Vert u\Vert _{\mu }^{2}-\frac{\Vert g\Vert _{\infty }|\{V<c_{0}\}|^{(6-p)/6}%
}{pS^{p}}\Vert u\Vert _{\mu }^{p} \\
& \geq -\frac{\lambda \Vert f\Vert _{3/2}}{2S^{2}}\rho _{a,\lambda }^{2}-%
\frac{\Vert g\Vert _{\infty }|\{V<c_{0}\}|^{(6-p)/6}}{pS^{p}}\rho
_{a,\lambda }^{p},
\end{align*}%
which implies that $-\infty <\beta _{\mu }<0$. By the Ekeland variational
principle \cite{E} and $J_{a,\lambda }^{\mu }(u)=J_{a,\lambda }^{\mu
}(\left\vert u\right\vert )$, there exists a $(PS)_{\beta _{\mu }}$-sequence
$\{u_{n}\}\subset B_{\rho _{a,\lambda }}$ with $u_{n}\geq 0$ in $\mathbb{R}%
^{3}.$ Then by Lemma \ref{PS-C}, there exists $0\leq u_{0}^{\left( 2\right)
}\in X$ such that $J_{a,\lambda }^{\mu }(u_{0}^{\left( 2\right) })=\beta
_{\mu }<0$ and $(J_{a,\lambda }^{\mu })^{\prime }(u_{0}^{\left( 2\right)
})=0 $ for $\mu $ sufficiently large, this implies that $u_{0}^{\left(
2\right) }$ is a nontrivial nonnegative solution of Eq. $\left( K_{a,\lambda
}^{\mu }\right) .$ The strong Maximum Principle implies that $u_{0}^{\left(
2\right) }>0$ in $\mathbb{R}^{3}.$ Consequently, this completes the proof of
Theorem \ref{T1}.

\section{Proofs of Theorems \protect\ref{T2} and \protect\ref{T3}}

We start this section by showing that the functional $J_{a,\lambda }^{\mu }$
is coercive and bounded below on $X$ when $N\geq 3$ and $2<p<\min
\{4,2^{\ast }\}.$

\begin{lemma}
\label{L5} Suppose that $N\geq 3,2<p<\min \{4,2^{\ast }\}$ and conditions $%
(V1)-(V2),(D1)-(D2)$ hold. In addition, we assume that condition $(H3)$
holds for $N=3$. Then for each $a>0$ and $\lambda >0$,
\begin{equation*}
J_{a,\lambda }^{\mu }(u)\geq \frac{1}{4}\Vert u\Vert _{\mu
}^{2}-C_{N,a,\lambda }
\end{equation*}%
for $\mu $ sufficiently large, where the number $C_{N,a,\lambda }>0$ is
independent of $\mu $.
\end{lemma}

\begin{proof}
By condition $(D1)$, the H\"{o}lder and Young's inequalities, we have
\begin{equation*}
\frac{\lambda }{2}\int_{\mathbb{R}^{N}}f(x)u^{2}dx\leq \Vert f\Vert
_{N/2}S^{-2}\Vert u\Vert _{D^{1,2}}^{2}\leq \frac{a}{12}\Vert u\Vert
_{D^{1,2}}^{4}+\frac{3\lambda ^{2}\Vert f\Vert _{N/2}^{2}}{4aS^{4}}.
\end{equation*}%
Then it holds%
\begin{align*}
J_{a,\lambda }^{\mu }(u)& =\frac{a}{4}\Vert u\Vert _{D^{1,2}}^{4}+\frac{1}{2}%
\Vert u\Vert _{\mu }^{2}-\frac{\lambda }{2}\int_{\mathbb{R}^{N}}f(x)u^{2}dx-%
\frac{1}{p}\int_{\mathbb{R}^{N}}g(x)|u|^{p}dx \\
\ & \geq \frac{a}{4}\Vert u\Vert _{D^{1,2}}^{4}+\frac{1}{2}\Vert u\Vert
_{\mu }^{2}-\left( \frac{a}{12}\Vert u\Vert _{D^{1,2}}^{4}+\frac{3\lambda
^{2}\Vert f\Vert _{N/2}^{2}}{4aS^{4}}\right) -\frac{1}{p}\int_{\mathbb{R}%
^{N}}g(x)|u|^{p}dx \\
\ & =\frac{1}{2}\Vert u\Vert _{\mu }^{2}+\frac{a}{6}\Vert u\Vert
_{D^{1,2}}^{4}-\frac{1}{p}\int_{\mathbb{R}^{N}}g(x)|u|^{p}dx-\frac{3\lambda
^{2}\Vert f\Vert _{N/2}^{2}}{4aS^{4}}.
\end{align*}%
Next, about the estimate of the nonlinear term $\frac{1}{p}\int_{\mathbb{R}%
^{N}}g(x)|u|^{p}dx$, we consider two cases as follows.\newline
Case $(i):N=3$. Using the Sobolev, H\"{o}lder and Caffarelli-Kohn-Nirenberg
inequalities and condition $\left( H3\right) $ gives
\begin{align}
& \frac{1}{p}\int_{\mathbb{R}^{3}}g(x)\left\vert u\right\vert ^{p}dx  \notag
\\
& \leq \frac{1}{p}\left( \int_{\{g>0\}}g(x)^{\frac{4}{6-p}}u^{2}dx\right) ^{%
\frac{6-p}{4}}\left( \int_{\mathbb{R}^{3}}u^{6}dx\right) ^{\frac{p-2}{4}}
\notag \\
& \leq \frac{1}{p}\left( \int_{\{g>0\}\cap \{\left\vert x\right\vert
>R_{\ast }\}}g(x)^{\frac{4}{6-p}}u^{2}dx+\int_{\{g>0\}\cap B_{R_{\ast
}}(0)}g(x)^{\frac{4}{6-p}}u^{2}dx\right) ^{\frac{6-p}{4}}\left( \frac{\Vert
u\Vert _{D^{1,2}}^{6}}{S^{6}}\right) ^{\frac{p-2}{4}}  \notag \\
& \leq \frac{1}{p}\left[ \left( c_{\ast }\Vert g\Vert _{\infty }\right) ^{%
\frac{2}{6-p}}\left( \int_{\mathbb{R}^{3}}V(x)u^{2}dx\right) ^{\frac{2\left(
4-p\right) }{6-p}}\left( \int_{\mathbb{R}^{3}}\frac{u^{2}}{\left\vert
x\right\vert ^{2}}dx\right) ^{\frac{p-2}{6-p}}+\frac{\Vert g\Vert _{\infty
}^{\frac{4}{6-p}}\left\vert B_{R_{\ast }}(0)\right\vert ^{\frac{2}{3}}}{S^{2}%
}\Vert u\Vert _{D^{1,2}}^{2}\right] ^{\frac{6-p}{4}}  \notag \\
& \text{ \ }\cdot \left( \frac{\Vert u\Vert _{D^{1,2}}^{6}}{S^{6}}\right) ^{%
\frac{p-2}{4}}  \notag \\
& \leq \frac{2^{\frac{6-p}{4}}\left( c_{\ast }\Vert g\Vert _{\infty }\right)
^{\frac{1}{2}}\overline{C}_{0}^{\frac{p-2}{2}}}{pS^{\frac{3(p-2)}{2}}}\left(
\int_{\mathbb{R}^{3}}V(x)u^{2}dx\right) ^{\frac{4-p}{2}}\Vert u\Vert
_{D^{1,2}}^{2(p-2)}+\frac{2^{\frac{6-p}{4}}\Vert g\Vert _{\infty }\left\vert
B_{R_{\ast }}(0)\right\vert ^{\frac{6-p}{6}}}{pS^{p}}\Vert u\Vert
_{D^{1,2}}^{p},  \label{3-1}
\end{align}%
where $\overline{C}_{0}$ is the sharp constant of Caffarelli-Kohn-Nirenberg
inequality. Moreover, using Young's inequality yields
\begin{align}
& \frac{2^{\frac{6-p}{4}}\left( c_{\ast }\Vert g\Vert _{\infty }\right) ^{%
\frac{1}{2}}\overline{C}_{0}^{\frac{p-2}{2}}}{pS^{\frac{3(p-2)}{2}}}\left(
\int_{\mathbb{R}^{3}}V(x)u^{2}dx\right) ^{\frac{4-p}{2}}\Vert u\Vert
_{D^{1,2}}^{2(p-2)}  \notag \\
& \leq \frac{a}{12}\Vert u\Vert _{D^{1,2}}^{4}+(4-p)\left( \frac{c_{\ast
}\Vert g\Vert _{\infty }}{p^{2}}\right) ^{\frac{1}{4-p}}\left( \frac{6\sqrt{2%
}(p-2)\overline{C}_{0}}{aS^{3}}\right) ^{\frac{p-2}{4-p}}\int_{\mathbb{R}%
^{3}}V(x)u^{2}dx  \label{3-2}
\end{align}%
and
\begin{align}
& \frac{2^{\frac{6-p}{4}}\Vert g\Vert _{\infty }\left\vert B_{R_{\ast
}}(0)\right\vert ^{\frac{6-p}{6}}}{pS^{p}}\Vert u\Vert _{D^{1,2}}^{p}  \notag
\\
& \leq \frac{a}{12}\Vert u\Vert _{D^{1,2}}^{4}+\frac{4-p}{p}2^{\frac{p-2}{4-p%
}}|B_{R_{\ast }}(0)|^{\frac{12-2p}{12-3p}}\left( \frac{\Vert g\Vert _{\infty
}}{S^{p}}\right) ^{\frac{4}{4-p}}\left( \frac{3}{a}\right) ^{\frac{p}{4-p}}.
\label{3-3}
\end{align}%
It follows from $(\ref{3-1})-(\ref{3-3})$ that
\begin{align*}
J_{a,\lambda }^{\mu }(u)& \geq \frac{1}{2}\Vert u\Vert _{\mu
}^{2}-(4-p)\left( \frac{c_{\ast }\Vert g\Vert _{\infty }}{p^{2}}\right) ^{%
\frac{1}{4-p}}\left( \frac{6\sqrt{2}(p-2)\overline{C}_{0}}{aS^{3}}\right) ^{%
\frac{p-2}{4-p}}\int_{\mathbb{R}^{3}}V(x)u^{2}dx-C_{3,a,\lambda } \\
& \geq \frac{1}{4}\Vert u\Vert _{\mu }^{2}-C_{3,a,\lambda } \text{ for all } \mu \geq \mu_{0},
\end{align*}%
where
\begin{equation*}
\mu_{0} :=4(4-p)\left( \frac{c_{\ast }\Vert g\Vert _{\infty }}{p^{2}}\right)
^{\frac{1}{4-p}}\left( \frac{6\sqrt{2}(p-2)\overline{C}_{0}}{aS^{3}}\right)
^{\frac{p-2}{4-p}}
\end{equation*}%
and
\begin{equation*}
C_{3,a,\lambda }:=\frac{4-p}{p}2^{\frac{p-2}{4-p}}|B_{R_{\ast }}(0)|^{\frac{%
12-2p}{12-3p}}\left( \frac{\Vert g\Vert _{\infty }}{S^{p}}\right) ^{\frac{4}{%
4-p}}\left( \frac{3}{a}\right) ^{\frac{p}{4-p}}+\frac{3\lambda ^{2}\Vert
f\Vert _{3/2}^{2}}{4aS^{4}}.
\end{equation*}%
Case $(ii):N\geq 4$. Using $(\ref{r1})$ gives
\begin{align}
& \frac{1}{p}\int_{\mathbb{R}^{N}}g(x)|u|^{p}dx  \notag \\
& \leq \frac{\Vert g\Vert _{\infty }}{p}\left[ \frac{|\{V<c_{0}\}|^{\frac{%
2^{\ast }-2}{2^{\ast }}}}{S^{2}}\Vert u\Vert _{D^{1,2}}^{2}+\frac{1}{c_{0}}%
\int_{\mathbb{R}^{N}}V(x)u^{2}dx\right] ^{\frac{2^{\ast }-p}{2^{\ast }-2}%
}\left( \frac{\Vert u\Vert _{D^{1,2}}^{2^{\ast }}}{S^{2^{\ast }}}\right) ^{%
\frac{p-2}{2^{\ast }-2}}  \notag \\
& \leq \frac{2^{\frac{2^{\ast }-p}{2^{\ast }-2}}\Vert g\Vert _{\infty
}|\{V<c_{0}\}|^{\frac{2^{\ast }-p}{2^{\ast }}}}{pS^{p}}\Vert u\Vert
_{D^{1,2}}^{p}+\frac{\Vert g\Vert _{\infty }}{p}\left( \frac{2}{c_{0}}\int_{%
\mathbb{R}^{N}}V(x)u^{2}dx\right) ^{\frac{2^{\ast }-p}{2^{\ast }-2}}\left(
\frac{\Vert u\Vert _{D^{1,2}}^{2^{\ast }}}{S^{2^{\ast }}}\right) ^{\frac{p-2%
}{2^{\ast }-2}}.  \label{3-4}
\end{align}%
Moreover, by Young's inequality one has
\begin{eqnarray}
&&\frac{2^{\frac{2^{\ast }-p}{2^{\ast }-2}}\Vert g\Vert _{\infty
}|\{V<c_{0}\}|^{\frac{2^{\ast }-p}{2^{\ast }}}}{pS^{p}}\Vert u\Vert
_{D^{1,2}}^{p}  \notag \\
&\leq &\frac{a}{12}\Vert u\Vert _{D^{1,2}}^{4}+\frac{4-p}{4p}\left( \frac{2^{%
\frac{2^{\ast }-p}{2^{\ast }-2}}\Vert g\Vert _{\infty }|\{V<c_{0}\}|^{\frac{%
2^{\ast }-p}{2^{\ast }}}}{S^{p}}\right) ^{\frac{4}{4-p}}\left( \frac{3}{a}%
\right) ^{\frac{p}{4-p}},  \label{3-5}
\end{eqnarray}%
and
\begin{align}
& \frac{\Vert g\Vert _{\infty }}{p}\left( \frac{2}{c_{0}}\int_{\mathbb{R}%
^{N}}V(x)u^{2}dx\right) ^{\frac{2^{\ast }-p}{2^{\ast }-2}}\left( \frac{\Vert
u\Vert _{D^{1,2}}^{2^{\ast }}}{S^{2^{\ast }}}\right) ^{\frac{p-2}{2^{\ast }-2%
}}  \notag \\
& \leq \frac{a}{12}\Vert u\Vert _{D^{1,2}}^{4}+(4-p)\left( \frac{\Vert
g\Vert _{\infty }}{p}\right) ^{\frac{2}{4-p}}\left( \frac{6(p-2)}{aS^{4}}%
\right) ^{\frac{p-2}{4-p}}\frac{1}{c_{0}}\int_{\mathbb{R}^{4}}V(x)u^{2}dx
\label{3-6}
\end{align}%
for $N=4;$
\begin{align}
& \frac{\Vert g\Vert _{\infty }}{p}\left( \frac{2}{c_{0}}\int_{\mathbb{R}%
^{N}}V(x)u^{2}dx\right) ^{\frac{2^{\ast }-p}{2^{\ast }-2}}\left( \frac{\Vert
u\Vert _{D^{1,2}}^{2^{\ast }}}{S^{2^{\ast }}}\right) ^{\frac{p-2}{2^{\ast }-2%
}}  \notag \\
& \leq \frac{2(2^{\ast }-p)\Vert g\Vert _{\infty }}{(2^{\ast }-2)pc_{0}}%
\int_{\mathbb{R}^{N}}V(x)u^{2}dx+\frac{a}{12}\Vert u\Vert _{D^{1,2}}^{4}+%
\frac{4-2^{\ast }}{4}\left( \frac{(p-2)\Vert g\Vert _{\infty }}{(2^{\ast
}-2)pS^{2^{\ast }}}\right) ^{\frac{4}{4-2^{\ast }}}\left( \frac{32^{\ast }}{a%
}\right) ^{\frac{2^{\ast }}{4-2^{\ast }}}  \label{3-7}
\end{align}%
for $N\geq 5.$ We now set%
\begin{equation*}
\mu _{1}=\left\{
\begin{array}{ll}
4(4-p)c_{0}^{-1}\left( \frac{\Vert g\Vert _{\infty }}{p}\right) ^{\frac{2}{%
4-p}}\left( \frac{6(p-2)}{aS^{4}}\right) ^{\frac{p-2}{4-p}} & \text{ for }%
N=4, \\
\frac{8(2^{\ast }-p)\Vert g\Vert _{\infty }}{(2^{\ast }-2)pc_{0}} & \text{
for }N\geq 5,%
\end{array}%
\right.
\end{equation*}%
and
\begin{equation*}
C_{N,a,\lambda }=\left\{
\begin{array}{ll}
\frac{(4-p)|\{V<c_{0}\}|}{p}\left( \frac{\Vert g\Vert _{\infty }}{S^{p}}%
\right) ^{\frac{4}{4-p}}\left( \frac{3}{a}\right) ^{\frac{p}{4-p}}+\frac{%
3\lambda ^{2}\Vert f\Vert _{2}^{2}}{4aS^{4}} & \text{ for }N=4, \\
\begin{array}{l}
\frac{4-p}{4p}\left( \frac{2^{\frac{2^{\ast }-p}{2^{\ast }-2}}|\{V<c_{0}\}|^{%
\frac{2^{\ast }-p}{2^{\ast }}}\Vert g\Vert _{\infty }}{S^{p}}\right) ^{\frac{%
4}{4-p}}\left( \frac{3}{a}\right) ^{\frac{p}{4-p}} \\
+\frac{4-2^{\ast }}{4}\left( \frac{(p-2)\Vert g\Vert _{\infty }}{(2^{\ast
}-2)pS^{2^{\ast }}}\right) ^{\frac{4}{4-2^{\ast }}}\left( \frac{3\cdot
2^{\ast }}{a}\right) ^{\frac{2^{\ast }}{4-2^{\ast }}}+\frac{3\lambda
^{2}\Vert f\Vert _{N/2}^{2}}{4aS^{4}}%
\end{array}
& \text{ for }N\geq 5.%
\end{array}%
\right.
\end{equation*}%
Thus, it follows from $(\ref{3-4})-(\ref{3-7})$ that
\begin{equation}
J_{a,\lambda }^{\mu }(u)\geq \frac{1}{4}\Vert u\Vert _{\mu
}^{2}-C_{N,a,\lambda }\text{ for all }\mu \geq \mu _{1}.  \notag  \label{CB4}
\end{equation}%
Consequently, this completes the proof.
\end{proof}

\textbf{We are now ready to prove Theorem \ref{T2}:} $(i)$ By Lemma \ref{L2}
$(i)$, for each $0<a<a_{0}(p)$ and $0<\lambda <\lambda _{1}(f_{\Omega })$,
the functional $J_{a,\lambda }^{\mu }$ has the mountain pass geometry for $%
\mu $ sufficiently large. Let
\begin{equation*}
\alpha _{\mu }:=\inf_{\gamma \in \Gamma }\max_{0\leq s\leq 1}J_{a,\lambda
}^{\mu }(\gamma (s))\quad \text{with}\quad \Gamma =\left\{ \gamma \in
C\left( [0,1],X\right) \ |\ \gamma (0)=0,\gamma (1)=e_{0}\right\} ,
\end{equation*}%
where $e_{0}$ is as in Lemma \ref{L2}. It is evident that
\begin{equation}
0<\alpha _{\mu }\leq \max_{0\leq s\leq 1}J_{a,\lambda }^{\mu }(se_{0})=:D_{0}
\label{5-1}
\end{equation}%
and that $D_{0}$ is independent of $\mu $ since $e_{0}\in H_{0}^{1}(\Omega )$%
. Let $\{u_{n}\}$ be a $(PS)_{\alpha _{\mu }}$-sequence, that is $%
J_{a,\lambda }^{\mu }(u_{n})=\alpha _{\mu }+o(1)$ and $\left( J_{a,\lambda
}^{\mu }\right) ^{\prime }(u_{n})=o(1)$. In fact, since $J_{a,\lambda }^{\mu
}(u_{n})=J_{a,\lambda }^{\mu }(\left\vert u_{n}\right\vert )$ for all $n,$
we may assume that $u_{n}\geq 0.$ Moreover, by Lemma \ref{L5} and $(\ref{5-1}%
)$, we deduce that there exists $d_{0}>0$ such that the $(PS)_{\alpha _{\mu
}}$-sequence $\{u_{n}\}$ satisfies $\Vert u_{n}\Vert _{\mu }<d_{0}$ for $\mu
$ sufficiently large, which implies that the functional $J_{a,\lambda }^{\mu
}$ satisfies the $(PS)_{\alpha _{\mu }}$-condition via Lemma \ref{PS-C}.
Therefore, there exists $0\leq u_{0}^{\left( 3\right) }\in X$ such that $%
J_{a,\lambda }^{\mu }(u_{0}^{\left( 3\right) })=\alpha _{\mu }>0$ and $%
(J_{a,\lambda }^{\mu })^{\prime }(u_{0}^{\left( 3\right) })=0$ for $\mu $
sufficiently large, this implies that $u_{0}^{\left( 3\right) }$ is a
nontrivial nonnegative solution of Eq. $\left( K_{a,\lambda }^{\mu }\right)
. $ The strong Maximum Principle implies that $u_{0}^{\left( 3\right) }>0$
in $\mathbb{R}^{N}.$

Next, we consider the infimum of $J_{a,\lambda }^{\mu }$ on the set $\{u\in
X\ |\ \Vert u\Vert _{\mu }\geq \overline{\rho }_{a,\lambda }\}$ with $%
\overline{\rho }_{a,\lambda }$ as given in Lemma \ref{L2} $(i)$. Set
\begin{equation*}
\beta _{\mu }=\inf_{\Vert u\Vert _{\mu }\geq \overline{\rho }_{a,\lambda
}}J_{a,\lambda }^{\mu }(u).
\end{equation*}%
By virtue of $\Vert e_{0}\Vert _{\mu }>\overline{\rho }_{a,\lambda
},J_{a,\lambda }^{\mu }(e_{0})<0$ and Lemma \ref{L5}, we conclude that $%
-C_{N,a,\lambda }<\beta _{\mu }<0$. It follows from $J_{a,\lambda }^{\mu
}(u)=J_{a,\lambda }^{\mu }(\left\vert u\right\vert ),$ the Ekeland
variational principle \cite{E} and Lemma \ref{L5} that there exists a
bounded $(PS)_{\beta _{\mu }}$-sequence $\{u_{n}\}\subset X$ with $u_{n}\geq
0$ in $\mathbb{R}^{N}.$ Hence, by Lemma \ref{PS-C}, there exists $0\leq
u_{0}^{\left( 4\right) }\in X$ with $\Vert u_{0}^{\left( 4\right) }\Vert
_{\mu }\geq \overline{\rho }_{a,\lambda }$ such that $J_{a,\lambda }^{\mu
}(u_{0}^{\left( 4\right) })=\beta _{\mu }<0$ and $(J_{a,\lambda }^{\mu
})^{\prime }\left( u_{0}^{\left( 4\right) }\right) =0$ for $\mu $
sufficiently large, this implies that $u_{0}^{\left( 4\right) }$ is a
nontrivial nonnegative solution of Eq. $\left( K_{a,\lambda }^{\mu }\right)
. $ The strong Maximum Principle implies that $u_{0}^{\left( 4\right) }>0$
in $\mathbb{R}^{N}.$

$(ii)$ Using Lemma \ref{L2} $(ii)$ and repeating the same argument as in the
proof of part $(i)$, we conclude that there exist two positive solutions $%
u_{0}^{\left( 3\right) }$ and $u_{0}^{\left( 4\right) }$ with $J_{a,\lambda
}^{\mu }(u_{0}^{\left( 3\right) })>0>J_{a,\lambda }^{\mu }(u_{0}^{\left(
4\right) })$ and $\Vert u_{0}^{\left( 4\right) }\Vert _{\mu }\geq \overline{%
\rho }_{a,\lambda }$ whenever $\lambda _{1}(f_{\Omega })\leq \lambda
<\lambda _{1}(f_{\Omega })+\overline{\delta }_{a}$, where $\overline{\delta }%
_{a}$ is as in Lemma \ref{L2} $(ii)$. To complete the proof of part $(ii)$,
for $\lambda _{1}(f_{\Omega })<\lambda <\lambda _{1}(f_{\Omega })+\overline{%
\delta }_{a}$, we consider the infimum of $J_{a,\lambda }^{\mu }$ on the
closed ball
\begin{equation*}
B_{\overline{\rho }_{a,\lambda }}:=\{u\in X\ |\ \Vert u\Vert _{\mu }\leq
\overline{\rho }_{a,\lambda }\}
\end{equation*}%
with $\overline{\rho }_{a,\lambda }$ as in Lemma \ref{L2} $(ii)$. Note that $%
\overline{\rho }_{a,\lambda }$ is independent of $\mu $. Set
\begin{equation*}
\eta _{\mu }=\inf_{\Vert u\Vert _{\mu }\leq \overline{\rho }_{a,\lambda
}}J_{a,\lambda }^{\mu }(u).
\end{equation*}%
For any $t>0$, we deduce that%
\begin{equation*}
J_{a,\lambda }^{\mu }(t\phi _{1})=-\frac{\lambda -\lambda _{1}(f_{\Omega })}{%
2}t^{2}+\frac{|\int_{\Omega }g(x)\phi _{1}^{p}dx|}{p}t^{p}+\frac{a\lambda
_{1}^{2}(f_{\Omega })}{4}t^{4}.
\end{equation*}%
Clearly, there exists $t_{0}>0$ such that $\Vert t_{0}\phi _{1}\Vert _{\mu }<%
\overline{\rho }_{a,\lambda }$ and $J_{a,\lambda }^{\mu }(t_{0}\phi _{1})<0$%
, which implies that $-\infty <\eta _{\mu }<0$. By $J_{a,\lambda }^{\mu
}(u)=J_{a,\lambda }^{\mu }(\left\vert u\right\vert )$ and the Ekeland
variational principle \cite{E}, there exists a $(PS)_{\eta _{\mu }}$%
-sequence $\{u_{n}\}\subset B_{\overline{\rho }_{a,\lambda }}$ with $%
u_{n}\geq 0$ in $\mathbb{R}^{N}.$ Then it follows from Lemma \ref{PS-C} that
there exists $0\leq u_{0}^{\left( 5\right) }\in X$ with $\Vert u_{0}^{\left(
5\right) }\Vert _{\mu }<\overline{\rho }_{a,\lambda }$ such that $%
J_{a,\lambda }^{\mu }(u_{0}^{\left( 5\right) })=\eta _{\mu }<0$ and $%
(J_{a,\lambda }^{\mu })^{\prime }(u_{0}^{\left( 5\right) })=0$ for $\mu $
sufficiently large, this implies that $u_{0}^{\left( 5\right) }$ is a
nontrivial nonnegative solution of Eq. $\left( K_{a,\lambda }^{\mu }\right)
. $ The strong Maximum Principle implies that $u_{0}^{\left( 5\right) }>0$
in $\mathbb{R}^{N}.$ Consequently, we complete the proof of Theorem \ref{T2}.

\textbf{At the end of this section, we give the proof of Theorem \ref{T3}:}
By virtue of Lemma \ref{L3}, we can easily reach the conclusion by using the
similar argument of Theorem \ref{T2} $(i).$ We omit it here.

\section*{Acknowledgments}

J. Sun was supported by the National Natural Science Foundation of China
(Grant No. 11671236), K.-H Wang was supported in part by the
Ministry of Science and Technology, Taiwan (Grant No. 108-2811-M-390-500) and T.-F. Wu was supported in part by the
Ministry of Science and Technology, Taiwan (Grant No.
108-2115-M-390-007-MY2).

\end{document}